\documentclass[12pt,reqno]{amsart}
\usepackage[numbers,sort&compress]{natbib}
\usepackage{amsfonts,bbm}
\usepackage{amssymb,color}
\usepackage{amssymb}
\usepackage{fancyhdr}
\usepackage[titletoc]{appendix}
\usepackage{enumitem}
\usepackage{amsgen}
\usepackage{amscd}
\usepackage{amsmath}
\usepackage{amsthm}
\usepackage{mathrsfs}
\usepackage{cases}
\usepackage{setspace}

\usepackage{accents} % This can substitute perfectly the HARPOON package.

\usepackage{verbatim}

\usepackage[colorlinks=true]{hyperref}
\hypersetup{urlcolor=blue, citecolor=green, linkcolor=blue}

\usepackage[margin=3cm, a4paper]{geometry}

\newtheorem{thm}{Theorem}[section]
\newtheorem{cor}[thm]{Corollary}
\newtheorem{lem}[thm]{Lemma}
\newtheorem{prop}[thm]{Proposition}
\newtheorem{defn}[thm]{ \bf{Definition}}

\theoremstyle{remark}
\newtheorem{remark}[thm]{Remark}

\newcommand{\EQ}[1]{\begin{align*}\begin{split} #1 \end{split}\end{align*}}
\newcommand{\EQn}[1]{\begin{align}\begin{split} #1 \end{split}\end{align}}

\newcommand{\Del}[1]{}

\def\norm#1{\left\|#1\right\|}

\def\wt#1{\widetilde{#1}}

\def\wb#1{\overline{#1}}
\def\pd{\partial}

\newcommand{\ra}{{\rightarrow}}
\newcommand{\hra}{{\hookrightarrow}}

\def\lsm{\lesssim}

\newcommand{\R}{{\mathbb R}}

\newcommand{\F}{{\mathcal{F}}}

\newcommand{\E}{{\mathcal{E}}}

\newcommand{\U}{{\mathcal{U}}}

\newcommand{\TT}{{\mathcal{T}}}

\newcommand{\re}{{\mathrm{Re}}}
\newcommand{\im}{{\mathrm{Im}}}

\def\ep{\varepsilon}

\def\th{\theta}

\def\De{\Delta}

\def\ta{\tau}

\newcommand{\I}{\infty}

\numberwithin{equation}{section}
\begin{document}
	
	\title[NLS]{Global theory for NLS in the weighted spaces I: finite pseudo conformal energy}
	
	\subjclass[2010]{}
	
	\author{Yujin Guo}
	\address{(Y. Guo) Center for Applied Mathematics\\
		Tianjin University\\
		Tianjin 300072, China}
	\email{guoyujin1021@163.com}
	
	\author{Jia Shen}
	\address{(J. Shen) School of Mathematical Sciences and LPMC\\
		Nankai University\\
		Tianjin 300071, China}
	\email{shenjia@nankai.edu.cn}
	
	\author{Changping Yang}
	\address{(C. Yang) Center for Applied Mathematics\\
		Tianjin University\\
		Tianjin 300072, China}
	\email{cp\_yang@tju.edu.cn}
	
	\date{}
	
	\begin{abstract}\noindent
		We intend to prove that for defocusing nonlinear Schr\"odinger equations, solutions with finite pseudo conformal energy must be global and scatter. 
		
		This fact is first proved by Bourgain \cite{1999-Bourgain} when $0<s_c<1$: the equation is locally well-posed in $H_x^{s_c}$, and if one further assumes that the initial data satisfies $xu_0\in L_x^2$, then the solution is global and scatters. The $s_c<0$ case remains not studied, and the best result is given by Beceanu, Deng, Soffer, and Wu \cite{2021-Beceanu-Deng-Soffer-Wu}: If the initial data is in $H_x^{s_c}$, radial, and compactly supported, then the solution is globally well-posed.
		
		In this paper, we extend Bourgain's result to the $s_c<0$ case. We prove that if the initial data satisfies $|x|^{-s_c}u_0\in L_x^2$ and $xu_0\in L_x^2$, or if the radial initial data  satisfies $u_0\in \dot{H}_x^{s_c}$ and $xu_0\in L_x^2$, then the solution is global and scatters. Bourgain's result is based on the subcritical a priori control on $L_x^{p+2}$, while our argument relies on the supercritical bound on $\F\dot H_x^1$.
	\end{abstract}
	
	\maketitle
	
	\setcounter{tocdepth}{1}
	
	\tableofcontents

	%\vspace{20pt}
	
	\section{Introduction}
	
	\vspace{0.5cm}
	
	We consider the Cauchy problem for the defocusing nonlinear Schr\"odinger equation (NLS):
	\begin{align}\label{NLS}
		\left\{
		\begin{aligned}
			&iu_t+\frac{1}{2}\Delta u=|u|^p u, \\
			&u(0,x)=u_0(x),
		\end{aligned}
		\right.
	\end{align}
	with $p>0$ and $u(t,x):\mathbb{R}\times\mathbb{R}^d\rightarrow \mathbb{C}$. The equation \eqref{NLS} is invariant under the scaling transform:
	\begin{align}\label{illposedness-scale transform}
		u(t,x)\rightarrow u_{\lambda}(t,x)=\lambda^{-\frac{2}{p}} u(\lambda^{-2} t, \lambda^{-1} x).
	\end{align}
	Denote the scaling critical exponent
	\begin{align*}
		s_c:=\frac{d}{2}-\frac{2}{p},
	\end{align*}
	then the scaling transform leaves  $\dot{H}_x^{s_{c}}$-norm invariant, that is, $\|u(0)\|_{\dot H_x^{s_{c}}}=\|u_{\lambda}(0)\|_{\dot H_x^{s_{c}}}$. Define the homogeneous weighted Sobolev space $\mathcal{F} \dot H_x^{s}(\mathbb{R}^d):= L_x^2(\mathbb{R}^d;|x|^{2s}\mathrm{d} x)$ and the inhomogeneous space $\mathcal{F} H_x^{s}(\mathbb{R}^d):= L_x^2(\mathbb{R}^d;\langle x\rangle ^{2s}\mathrm{d} x)$ for some $s>0$.
	
	Define the pseudo conformal transform as
	\begin{align}
		\mathcal{T} f(t,x) := \frac{1}{(it)^{d/2}} e^{\frac{i|x|^2}{2t}} \bar{f}\big(\frac{1}{t},\frac{x}{t}\big).
	\end{align}
	Then the inverse pseudo conformal transform satisfies $\mathcal{T}^{-1}=\mathcal{T}$. If $u$ is the solution of \eqref{NLS} with initial data $u_0$, then for any $t>0$, $\mathcal{U}=\mathcal{T}u$ solves a nonlinear Schr\"odinger equation
	\begin{align}\label{CNLS-I}
		i \pd_t \mathcal{U}+\frac{1}{2}\Delta \mathcal{U}=t^{\frac{dp}{2}-2}|\mathcal{U}|^p \mathcal{U},
	\end{align}
	with $S(-t)\mathcal{U}$ tending to $\mathcal{F}^{-1}\bar{u}_0$ as $t\rightarrow +\infty$ in a suitable sense. For \eqref{CNLS-I}, we define the pseudo conformal energy
	\begin{align}\label{Pseudo-conformal-energy}
		\mathcal{E}(t):=\frac{1}{4}t^{2-\frac{dp}{2}}\int_{\mathbb{R}^d}|\nabla\mathcal{U}|^2\mathrm{d}x+\frac{1}{p+2}\int_{\mathbb{R}^d}|\mathcal{U}|^{p+2}\mathrm{d}x.
	\end{align}
	Then we obtain 
	\begin{align}\label{derivate for P-c-e}
		\dfrac{\mathrm{d}}{\mathrm{d}t}\mathcal{E}(t)=\frac{1}{4}\big(2-\frac{dp}{2}\big)t^{1-\frac{dp}{2}}\int_{\mathbb{R}^d}|\nabla\mathcal{U}|^2\mathrm{d}x.
	\end{align}

	\subsection{Local theory in the homogeneous weighted space}
	
	In this paper, we study the NLS with initial data in a weighted space, rather than a Sobolev space of negative regularity.  For data in negative regularity Sobolev spaces, there are well-posedness results available in the radial setting \cite{2013-Cho-Hwang-Ozawa,2008-Hidano}; however, some \textit{ill-posedness} results hold for non-radial data in $H_x^s(\mathbb{R}^d)$ whenever $s<0$ \cite{2008-Christ-Colliander-Tao}.
	
	We first present some local results in the homogeneous weighted space. The local well-posedness and small data global well-posedness and scattering in the critical space $\mathcal{F}\dot{H}^{-s_c}(\R^d)$ have been proved by \cite{2015-Masaki,2017-Killip-Masaki-Murphy-Visan}. Now, we extend these results to the fully subcritical, critical, and supercritical cases:
	\begin{thm}[Local theory in the homogeneous weighted space]
		\label{thm:local-s-order-initialdata}
		Assume that $d\ge1$, $0<s<\min \{1, \frac{d}{2}\}$, and $u_0\in\F\dot H_x^s(\R^d)$. Then, the following statements hold:
		\begin{enumerate}
			\item Assume further that $d\leq 6$. Let $\frac{4s}{d+2s}<p\le \frac{4}{d+2s}$. Then there exists some $T> 0$ such that \eqref{NLS} admits a unique solution $u$ with $e^{-\frac12it\De}u\in C([0,T];\F\dot H_x^s(\R^d))$.
			\item Assume further that $d+2s>4$. Let  $\frac{4}{d+2s}<p\le\min\{1, \frac{4}{d-2}\}$. Then \eqref{NLS} is ill-posed.
		\end{enumerate}
	\end{thm}
	\begin{remark} We make several remarks regarding this result:
		\begin{enumerate}
			\item 
			The precise statements of the well-posedness and ill-posedness are given in Sections \ref{section-lwp} and \ref{section-illposed}. %Moreover,  in \cite{2015-Masaki}, Masaki proved the equation \eqref{NLS} is local well posedness when $p=\frac{4}{d+2s}$.
			\item
			The main purpose of this theorem is to identify the criticality of the equation in the weighted space. Note that for fixed $s$, $s_c=-s$ if $p=\frac{4}{d+2s}$, $s_c<-s$ if $p<\frac{4}{d+2s}$, and $s_c>-s$ if $p>\frac{4}{d+2s}$. Thus, $\F\dot H_x^{-s_c}(\R^d)$ is the critical space for the well-posedness theory of NLS. In $\F\dot H_x^s(\R^d)$, we can say that the NLS is $\F\dot H_x^s$-\textit{subcritical} when $p<\frac{4}{d+2s}$, $\F\dot H_x^s$-\textit{critical} when $p=\frac{4}{d+2s}$, and $\F\dot H_x^s$-\textit{supercritical} when $p>\frac{4}{d+2s}$. %in view of the pseudo conformal energy, we can also say that the 3D NLS is \textit{pseudo conformal energy subcritical} when $p<\frac45$, \textit{pseudo conformal energy critical} when $p=\frac45$, and \textit{pseudo conformal energy supercritical} when $p>\frac45$.
			
			\item 
			The restriction $d\leq 6$ in Theorem \ref{thm:local-s-order-initialdata} (1) comes from some technical obstructions. One may follow the stability argument as in \cite{2015-Masaki,2005-Tao-Visan,2011-Li-Zhang,2013-Killip-Visan} to remove it, and we do not pursue this issue here. %However, the time of existence $T$ will depend on the profile of $u_0$, and not only on its norm.
			
			\item 
			Our ill-posedness result does not cover the range $p>1$. This is due to some technical difficulties, and the same restriction also appears in the previous study of ill-posedness in Sobolev space, see \cite{2003-Christ-Colliander-Tao}. %Thus, this leads to the technical assumption $d+2s>4$ in the ill-posedness part.
			\item 
			Recall that the scaling critical regularity plays an important role in the study of well-posedness theory. When $s_c>0$, it is well-known that the local well-posedness holds in $\dot H_x^{s}$ with $s\ge s_c$ (\cite{1989-Cazenave-Weissler}), while there exists some ill-posedness in $\dot H_x^{s}$ when $s<s_c$ (\cite{2003-Christ-Colliander-Tao}). The results in Theorem \ref{thm:local-s-order-initialdata} provide an analogous conclusion for $s_c<0$ under weighted assumptions: the local well-posedness should hold in $\F\dot H_x^{s}$ with $s\le-s_c$, while there exists some ill-posedness in $\F\dot H_x^{s}$ when $s>-s_c$. This also provides an interesting observation: local well-posedness may fail even for more localized data.
		\end{enumerate}
		
	\end{remark}
	
	In order to further extend the local solution obtained in Theorem \ref{thm:local-s-order-initialdata} globally, it is natural to consider the NLS with the Cauchy data $e^{-\frac12iT\De}u(T) \in \F\dot H_x^s(\R^d)$, rather than $u(T) \in \F\dot H_x^s(\R^d)$. This motivates us to consider the following Cauchy problem at $t_0\ne0$:
	\begin{align}\label{NLS-t_0}
		\left\{
		\begin{aligned}
			&iu_t+\frac{1}{2}\Delta u=|u|^p u, \\
			&u(t_0,x)=u_1(x),
		\end{aligned}
		\right.
	\end{align}
	with initial data condition $e^{-\frac12it_0\De}u_1\in \mathcal{F}\dot H_x^s(\R^d)$. We may assume that $t_0>0$; otherwise we can apply the transform $u(t,x)\ \ra\ \wb{ u(-t,x)}$.
	\begin{thm}[Local well-posedness away from the origin]
		\label{thm:local-1order-awayfromorigin}
		Let $1\leq d \leq 6$, $0<s<\min \{1, \frac{d}{2}\}$, and $t_0>0$. Assume also that $\frac{4s}{d+2s}<p < \frac{4}{d-2s}$ and $e^{-\frac12it_0\De}u_1\in \mathcal{F}\dot H_x^s(\R^d)$. Then, there exists some $T=T(\|e^{-\frac12it_0\De}u_1\|_{\mathcal{F}\dot H_x^s(\R^d)})$ with $0<T<t_0$ such that \eqref{NLS-t_0} admits a unique solution $u$ with $e^{-\frac12it\De}u\in C([t_0-T,t_0+T];\F\dot H_x^s(\R^d))$.
	\end{thm}
	
	\begin{remark}
		Compared with Theorem \ref{thm:local-s-order-initialdata}, this result shows that when considering the Cauchy problem in the weighted space away from the origin, one can obtain local well-posedness regardless of the scaling. This general fact is first observed in our previous result \cite{SW24quaNLS}. %Using the notion in this paper, 3D quadratic NLS is pseudo conformal energy supercritical.
	\end{remark}

	\subsection{Global well-posedness with finite pseudo conformal energy}
	Next, we turn to the study of the global theory of NLS.
	
	First, we consider the $s_c\ge0$ case. The NLS is known to be locally well-posed in $H_x^{s_c}$ or $\dot H_x^{s_c}$. A basic question is whether the local solution can be extended globally. It is well-known that the defocusing NLS is global in critical space when $s_c=0$ or $s_c=1$, in which case the critical norm is bounded by a conservation law. However, when $0<s_c<1$ or $s_c>1$, it is not clear if the critical norm has some a priori control. Thus, it seems difficult to study the global well-posedness of defocusing NLS in $H_x^{s_c}$ or $\dot H_x^{s_c}$ when $0<s_c<1$ or $s_c>1$. We refer the reader to the introduction of \cite{SW24quaNLS} for more detailed discussions.
	
	An alternative approach is to find some $X$ such that the defocusing NLS is global in $H_x^{s_c} \cap X$ or $\dot H_x^{s_c} \cap X$.  Bourgain \cite{1999-Bourgain} first gave a global result in the case when $X=\F\dot H_x^1$, which enables the use of pseudo conformal energy: When $0<s_c<1$, the NLS is locally well-posed in $H_x^{s_c}$; if the initial data additionally satisfy $xu_0\in L_x^2$, then the solution is global and scatters. Furthermore, for 3D cubic NLS, Dodson \cite{Dodson20critical} first gave the global well-posedness in the critical space $\dot W_x^{11/7,7/6}$, where $X=\dot W_x^{11/7,7/6}\subset \dot H_x^{s_c}$. Then in \cite{SW20cubicNLS}, we  obtained the global well-posedness and scattering in $\dot H_x^{1/2}$ intersected with $X=\dot W^{s,1}$ for $s>\frac{12}{13}$, in which case the mass is finite. Dodson \cite{Dod23IMRN} gave the global well-posedness and scattering merely in the critical space $\dot B_{1,1}^2$ with radial assumption, where the initial data can possess infinite mass.
	
	Next we turn to the $s_c<0$ case. Tsutsumi \cite{1987-Tsutsumi} showed that both the defocusing and focusing NLS admit global solutions in $L_x^2$, which is a subcritical space. As for the critical space, NLS is locally well-posed in $\F\dot H_x^{-s_c}$, or in $\dot H_x^{s_c}$ with radial data. For the global well-posedness in critical space, Beceanu, Deng, Soffer, and Wu \cite{2021-Beceanu-Deng-Soffer-Wu} extend the local solution with radial $\dot H_x^{s_c}$-data globally, additionally assuming that the initial data are compactly supported. In \cite{SW24quaNLS}, we have proved that the 3D quadratic NLS is global in critical weighted space $\F\dot H_x^{1/2}$ with radial data. For scattering in mass subcritical NLS, we refer the reader to the introduction of \cite{SW24quaNLS,SW23subNLS} for more detailed discussions.
	
	Note that Bourgain's result \cite{1999-Bourgain} illustrates a fact that, for the local solution of NLS, if we further assume that the solution possesses finite pseudo conformal energy, then the solution must be global and scatter. Now, in this paper, we intend to extend Bourgain's result to the $s_c<0$ case. For the mass subcritical NLS, previous results \cite{2015-Masaki,2017-Killip-Masaki-Murphy-Visan,2008-Hidano} show that the local well-posedness holds with one of the following initial data conditions:
	\begin{itemize}
		\item $u_0\in\F \dot H_x^{-s_c}$;
		\item $u_0\in\dot H_x^{s_c}$ and $u_0$ is radial.
	\end{itemize}
	\begin{comment}
		We first define two spaces:
		\begin{defn}
			(1) We say $u\in\mathcal{X}(\R)$, if
			\begin{align*}
				e^{-\frac{1}{2}it\Delta}u\in C(\mathbb{R}; \mathcal{F}\dot{H}^{-s_c}_x\cap\mathcal{F}\dot{H}^{1}_x(\mathbb{R}^3))\text{ and } |J(t)|^{-s_c} u \in L_t^{\frac{4p}{5p-4}} L_x^{\frac{3p}{2-p}} (\mathbb{R} \times \mathbb{R}^3).
			\end{align*}
			(2) We say $u\in\mathcal{Y}(\R)$, if
			\begin{align*}
				e^{-\frac12it\De}u\in C(\mathbb{R};\dot{H}_x^{s_c}\cap\mathcal{F}\dot{H}_x^1(\mathbb{R}^3))\text{ and }u\in L_t^{\frac{2p(p+2)}{4-p}}L_x^{p+2}(\mathbb{R}\times\mathbb{R}^3).
			\end{align*}
		\end{defn}
		\begin{remark}
			The operator $|J(t)|^{s}=e^{-\frac12it\De}|x|^{s}e^{-\frac12it\De}$ is given in \eqref{2-|J(t)|^s}, and $(\frac{4p}{5p-4}, \frac{3p}{2-p})$ is admissible pair. The pair $(q,r)=(\frac{2p(p+2)}{4-p}, p+2)$ is radial-admissible (See Definition \ref{defn:acceptable-admissible}) and satisfies
			\begin{align*}
				\frac{2}{q}+\frac{3}{r}=\frac{3}{2}-s_c.
			\end{align*}
		\end{remark}
	\end{comment}
	We extend the above two kinds of local solutions globally; thus, the global results are divided into two cases. The first one concerns the initial data in critical weighted space:
	\begin{thm}\label{main-theorem1}
		Let $1\leq p<\frac{4}{3}$. For any $u_0\in \mathcal{F}\dot{H}_x^{-s_c}\cap \mathcal{F}\dot{H}_x^1(\mathbb{R}^3)$, there exists a unique solution to \eqref{NLS} such that $e^{-\frac{1}{2}it\Delta}u\in C(\mathbb{R}; \mathcal{F}\dot{H}^{-s_c}_x\cap\mathcal{F}\dot{H}^{1}_x(\mathbb{R}^3))$. Moreover, the solution scatters in $\mathcal{F}\dot{H}_x^{-s_c}\cap \mathcal{F}\dot{H}_x^1(\mathbb{R}^3)$, that is, there exist $u_{\pm}\in \mathcal{F}\dot{H}_x^{-s_c}\cap \mathcal{F}\dot{H}_x^1(\R^3)$ such that 
		\begin{align*}
			\lim_{t\rightarrow \pm\infty}\|e^{-\frac12it\De}u(t)-u_{\pm}\|_{\mathcal{F}\dot{H}_x^{-s_c}\cap \mathcal{F}\dot{H}_x^1(\R^3)}=0.
		\end{align*}
	\end{thm}
	\begin{remark}
		\begin{enumerate}
			\item 
			The local well-posedness of NLS in $\F\dot H_x^{-s_c}(\R^3)$ is valid for $p<1$. We impose the restriction $p\ge1$, since the potential energy of pseudo conformal energy is required to be finite. It happens to equal the Strauss exponent $\gamma(d):=\frac{2-d+\sqrt{d^2+12d+4}}{2d}$ in 3D case.
			\item 
			In this theorem, $u_0\in\mathcal{F}\dot{H}_x^{-s_c}(\R^3)$ ensures the local well-posedness, and the assumption $u_0\in\mathcal{F}\dot{H}_x^{1}(\R^3)$ ensures the global well-posedness and scattering. By Theorem \ref{thm:local-s-order-initialdata}, $\mathcal{F}\dot{H}_x^{1}(\R^3)$ is a supercritical space for the equation. 
			\item 
			Traditionally, the global well-posedness is proved by some subcritical conservation law. In the proof of this theorem, we are able to give a global result utilizing the a priori control on $\norm{(x+it\nabla)u}_{L_x^2(\R^3)}$, which is supercritical, based on the observation in \cite{SW24quaNLS}. This is achieved in view of Theorem \ref{thm:local-1order-awayfromorigin}. 
			\item 
			We give the result in the 3D case as an example, and we believe that this result can be extended to other dimensions.
		\end{enumerate}
		
	\end{remark}
	For radial $\dot{H}_x^{s_c}(\R^3)$-data, we introduce the exponent $\gamma_1$ as the positive root of the equation $2x^2+15x-20=0$. Note that $1=\gamma(3)<\gamma_1 = 1.155\cdots$. Then, the global result is as follows:
	\begin{thm}\label{main-theorem2}
		Let $\gamma_1\leq p<\frac{4}{3}$. For any radial $u_0\in \dot{H}_x^{s_c}\cap\mathcal{F}\dot{H}_x^1(\mathbb{R}^3)$, there exists a unique solution to \eqref{NLS} such that $e^{-\frac12it\De}u\in C(\mathbb{R};\dot{H}_x^{s_c}\cap\mathcal{F}\dot{H}_x^1(\mathbb{R}^3))$. Moreover, the solution scatters in $\dot{H}_x^{s_c}\cap\mathcal{F}\dot{H}_x^1(\mathbb{R}^3)$, that is, there exist $u_{\pm}\in \dot{H}_x^{s_c}\cap\mathcal{F}\dot{H}_x^1(\R^3)$ such that 
		\begin{align*}
			\lim_{t\rightarrow \pm\infty}\|e^{-\frac12it\De}u(t)-u_{\pm}\|_{\dot{H}_x^{s_c}\cap\mathcal{F}\dot{H}_x^1(\mathbb{R}^3)}=0.
		\end{align*}  
	\end{thm}
	\begin{remark}
		The condition $\gamma_1\leq p$ comes from the restriction of  spacetime norm exponents in the radial Strichartz estimates (see Remark \ref{remark-radial-restriction for p} for more details). Beceanu, Deng, Soffer, and Wu's result \cite{2021-Beceanu-Deng-Soffer-Wu} requires compact support, while we impose the first weighted condition in Theorem \ref{main-theorem2}. Moreover, our result also includes the scattering result. 
	\end{remark}

	\vspace{2cm}
	
	\section{Preliminaries}\label{Preliminary}
	
	\vspace{0.5cm}
	
	\subsection{Basic notation}
	We use $\hat{f}$ or $\mathscr{F}f$ to denote the Fourier transform of $f$:
	\begin{align*}
		\hat{f}(\xi)=\mathscr{F}f(\xi):=(2\pi)^{-\frac{d}{2}}\int_{\mathbb{R}^d}f(x)e^{-ix\xi}\mathrm{d}x,
	\end{align*}
	and use $\mathscr{F}^{-1}g$ to denote the inverse Fourier transform of $g$:
	\begin{align*}
		\mathscr{F}^{-1}g(x):=(2\pi)^{-\frac{d}{2}}\int_{\mathbb{R}^d}g(\xi)e^{ix\xi}\mathrm{d}\xi.
	\end{align*}
	
	Let $e^{\frac12it\Delta}u_0$ be defined as the solution of 
	\begin{align}
		\left\{
		\begin{aligned}
			&iu_t+\frac{1}{2}\Delta u=0, \\
			&u(0,x)=u_0(x),
		\end{aligned}
		\right.
	\end{align}
	and let $S(t):=e^{\frac12it\Delta}$. Then, we have
	\begin{align*}
		S(t)u_0=\frac{1}{(2\pi it)^{d/2}} \int_{\mathbb{R}^d} e^{i\frac{|x-y|^2}{2t}} u_0(y)\mathrm{d}y = \mathscr{F}^{-1}e^{-i\frac{t|\xi|^2}{2}}\mathscr{F}u_0.
	\end{align*}
	
	\begin{defn}\label{defn:acceptable-admissible}
		Let $1\leq q,r\leq \infty$. Then we say that:
		\begin{itemize}
			\item $(q,r)$ is acceptable, if $1\leq q< \infty$ and $\frac{1}{q}<\frac{d}{2}-\frac{d}{r}$, or $(q,r)=(\infty, 2)$;
			\item $(q,r)$ is admissible, if $2\leq q,r \leq \infty$, $(q,r,d)\neq (2,\infty,2)$ and $\frac{2}{q}+\frac{d}{r}=\frac{d}{2}$;
			\item $(q,r)$ is radial-admissible, if $2\leq q,r \leq \infty$, $ (q,r)\neq(2,\frac{4d-2}{2d-3})$ and $\frac{2}{q}+\frac{2d-1}{r}\leq \frac{2d-1}{2}.$
		\end{itemize}
	\end{defn}
	
	In this paper, we will use Lorentz-modified spacetime norms. For an interval $I$, $1\leq q < \infty$, and $1\leq \alpha\leq \infty$, the Lorentz space $L_t^{q,\alpha}(I)$ is defined via the quasi-norm
	\begin{align*}
		\|f\|_{L_t^{q,\alpha}(I)} := q^{\frac{1}{\alpha}} \big\|\lambda|\{t\in I: |f(t)|>\lambda\}|^{\frac{1}{q}}\big\|_{L^{\alpha}((0,\infty), \frac{\mathrm{d}\lambda}{\lambda})}.
	\end{align*}
	We write $L_t^{q, \alpha}L_x^{r}(I\times\R^d)$ for the space of functions $u: I\times\R^d \rightarrow \mathbb{C}$ such that 
	\begin{align*}
		\|u\|_{L_t^{q, \alpha}L_x^{r}(I\times\R^d)} := \big\| \|u\|_{L_x^r(\R^d)} \big\|_{L_t^{q,\alpha}(I)} < \infty.
	\end{align*}
	
	\begin{comment}
		We define
		\begin{align*}%\label{def-S0(I)}
			\|f\|_{S^0(I)}&:=\mathop{\text{sup}}\{ \|f\|_{L_t^{q}L_x^r(I\times \mathbb{R}^d)} : (q,r) \text{ is  admissible} \},
		\end{align*}
		and
		\EQ{
			\|f\|_{S_*^0(I)}&:=\mathop{\text{sup}}\{ \|f\|_{L_t^{\infty}L_x^2 \cap L_t^{q,2}L_x^r(I\times \mathbb{R}^d)} : (q,r) \text{  is admissible and }q\ne\I \}.
		}
	\end{comment}

	\subsection{Pseudo conformal transform}
	Now, we define
	\begin{align}\label{def-M(t)}
		M(t) := e^{\frac{i|x|^2}{2t}}\text{, and } J(t):=x+it\nabla.
	\end{align}
	For $f(x): \mathbb{R}^d\rightarrow \mathbb{C}$, direct calculation yields
	\begin{align}\label{PCT-f(x)}
		\mathcal{T}S(t)f=S(t)\mathscr{F}^{-1}\bar{f},
	\end{align}
	and the following alternative representations of the vector field:
	\begin{align}\label{defn-J(t)}
		J(t)=S(t)xS(-t)=M(t)(it\nabla)M(-t).
	\end{align}
	We also define the fractional vector field for $s>0$:
	\begin{align}\label{2-|J(t)|^s}
		|J(t)|^s=S(t)|x|^sS(-t)=M(t)|t\nabla|^sM(-t).
	\end{align}
	Therefore, we have for a spacetime function $u(t,x)$,
	\begin{align}\label{PCT-u(t,x)}
		|\nabla|^s\mathcal{T}u=\mathcal{T}|J(t)|^su.
	\end{align}
	
	Now, we give two basic facts for the pseudo conformal transform:
	\begin{remark}
		For $\mathcal{U}=\mathcal{T}u$, by \eqref{PCT-u(t,x)}, we have that
		\begin{align*}
			\mathcal{U}\in C([a,b];\dot{H}_x^s(\mathbb{R}^d)) \quad &\Longleftrightarrow \quad |J(t)|^su\in C([\frac{1}{b},\frac{1}{a}];L_x^2(\mathbb{R}^d))\\
			&\Longleftrightarrow \quad S(-t)u\in C([\frac{1}{b},\frac{1}{a}];\mathcal{F}\dot{H}_x^s(\mathbb{R}^d)),
		\end{align*}
		for any $0<a<b<\infty$. Moreover, since $\mathcal{T}^{-1}=\mathcal{T}$, we also have
		\begin{align*}
			u \in C([a,b];\dot H_x^s(\R^d)) \quad \Longleftrightarrow & \quad |J(t)|^s\U \in C([\frac1b,\frac1a];L_x^2(\R^d)) \\
			\Longleftrightarrow & \quad S(-t)\U \in C([\frac1b,\frac1a];\F \dot H_x^s(\R^d)).
		\end{align*}
	\end{remark}
	\begin{remark}
		Let $\mathcal{U}=\mathcal{T}u$. Then for any $1\leq q, r\leq \infty$, we have that
		\begin{align}\label{PCT-remark-2}
			\|u(t,x)\|_{L_t^qL_x^r([a,b]\times\mathbb{R}^d)}=\|s^{\frac{d}{2}-\frac{2}{q}-\frac{d}{r}}\mathcal{U}(s,y)\|_{L_s^qL_y^r([\frac{1}{b}, \frac{1}{a}]\times\mathbb{R}^d)},
		\end{align}
		where $0<a<b<\infty.$
	\end{remark}

	\subsection{Useful lemmas}
	\begin{lem}[Kato-Ponce's inequality, \cite{2019-Li}]\label{2: Lemma-fractional chain rule 1}
		Let $0<s<1$, $1<p\leq\infty$, and $1< p_1, p_2, p_3, p_4 \leq \infty$ with $\frac{1}{p}=\frac{1}{p_1}+\frac{1}{p_2}$ and $\frac{1}{p}=\frac{1}{p_3}+\frac{1}{p_4}$. Then 
		\begin{align*}
			\big\| |\nabla|^s(fg) \big\|_{L^p}\lesssim \big\| |\nabla|^s f \big\|_{L^{p_1}}\|g\|_{L^{p_2}}+\big\| |\nabla|^s g \big\|_{L^{p_3}}\|f\|_{L^{p_4}}.
		\end{align*}
	\end{lem}
	
	\begin{lem}[Fractional chain rule, \cite{1991-Christ-Weinstein}]\label{2: Lemma-fractional chain rule 2}
		Suppose that $G\in C^1(\mathbb{C})$, $s\in (0,1]$, and $1<p, p_1, p_2 <\infty$ with $\frac{1}{p}=\frac{1}{p_1}+\frac{1}{p_2}$. Then 
		\begin{align*}
			\big\| |\nabla|^sG(u) \big\|_{L^p}\lesssim \|G^{\prime}(u)\|_{L^{p_1}}\big\| |\nabla|^s u \big\|_{L^{p_2}}.
		\end{align*}
	\end{lem}
	
	\begin{lem}[Fractional chain rule for H\"older continuous function, \cite{2007-Visan}]\label{2: Lemma-Holder continuous}
		Let $G$ be a H\"older continuous function of order $0<\alpha<1$. Then, for every $0<s<\alpha$, $1< p<\infty$, and $\frac{s}{\alpha}<\sigma<1$, we have
		\begin{align*}
			\big\| |\nabla|^sG(u) \big\|_{L^p} \lesssim \big\| |u|^{\alpha-\frac{s}{\sigma}} \big\|_{L^{p_1}} \big\| |\nabla|^{\sigma}u \big\|^{\frac{s}{\sigma}}_{L^{\frac{s}{\sigma}p_2}},
		\end{align*}
		provided $\frac{1}{p}=\frac{1}{p_1}+\frac{1}{p_2}$ and $(1 - \frac{s}{\alpha\sigma})p_1>1$.
	\end{lem}
	
	\begin{lem}[Some inequalities for the Lorentz space, \cite{2001-Nakanishi}]\label{2: Lemma-Lorentz space inequality}
		We have the refinement of H\"older's inequality
		\begin{align*}
			\|fg\|_{L^{p,q}}\lesssim\|f\|_{L^{p_1,q_1}}\|g\|_{L^{p_2,q_2}},
		\end{align*}
		where $\frac{1}{p}=\frac{1}{p_1}+\frac{1}{p_2}$, $\frac{1}{q}=\frac{1}{q_1}+\frac{1}{q_2}$, and $ p$, $p_1$, $p_2 < \infty$. Similarly, we have the refinement of Young's inequality
		\begin{align*}
			\|f*g\|_{L^{p,q}}\lesssim\|f\|_{L^{p_1,q_1}}\|g\|_{L^{p_2,q_2}},
		\end{align*}
		where $\frac{1}{p}=\frac{1}{p_1}+\frac{1}{p_2}-1$, $\frac{1}{q}=\frac{1}{q_1}+\frac{1}{q_2}$, $1<p, p_1, p_2 < \infty$ and $1\leq q, q_1, q_2 \leq \infty$.
	\end{lem}
	
	\begin{lem}[Strichartz estimates, \cite{2001-Nakanishi,1998-Keel-Tao}]
		Let $a\in I\subset \mathbb{R}$. For any admissible pairs $(q,r)$ and $(\wt{q},\wt{r})$,
		\begin{align*}
			\|S(t)f\|_{L_t^{\infty}L_x^2 \cap L_t^{q}L_x^r(I\times \mathbb{R}^d)}&\lesssim\|f\|_{L_x^2(\mathbb{R}^d)},\\
			\left\|\int_{a}^{t}S(t-\tau)F(\tau)\mathrm{d} \tau \right\|_{L_t^{\infty}L_x^2 \cap L_t^{q}L_x^{r}(I\times \mathbb{R}^d)} &\lesssim \|F \|_{L_t^{\wt{q}^{\prime}} L_x^{\wt{r}^{\prime}}(I\times \mathbb{R}^d)}.
		\end{align*}
		The above estimates also hold when $L_t^q$ and $L_t^{\wt q'}$ are replaced by Lorentz spaces $L_t^{q,2}$ and $L_t^{\wt q',2}$ with $q,\wt q\ne\I$.
	\end{lem}
	
	\begin{lem}[Inhomogeneous Strichartz estimates, \cite{2005-Foschi,2007-Vilela}]\label{2: Lemma-inhomogeneous Strichartz estimates}
		Let $I\subset \mathbb{R}$. Suppose that $(q,r)$ and $(\wt{q}, \wt{r})$ are acceptable pairs and satisfy the scaling condition
		\begin{align*}
			\frac{2}{q} + \frac{d}{r} +\frac{2}{\wt{q}} + \frac{d}{\wt{r}} =d.
		\end{align*}
		Then, for $a\in I$, we have
		\begin{align*}
			\left\|\int_{a}^{t}S(t-\tau)F(\tau)\mathrm{d}\tau\right\|_{L_t^{q}L_x^r(I\times \mathbb{R}^d)}\lesssim \|F\|_{L_t^{\wt{q}^{\prime}} L_x^{\wt{r}^{\prime}}(I\times \mathbb{R}^d)},
		\end{align*}
		provided that $d, q, r, \wt{q}, \wt{r}$ satisfy one of the following conditions:
		\begin{itemize}
			\item $d=1$, $2\leq r, \wt{r} \leq \infty$; 
			\item $d=2$, $2\leq r, \wt{r} < \infty$;
			\item $d\geq 3$, $\frac{1}{q} + \frac{1}{\wt{q}}<1$, $\frac{d-2}{d} \leq \frac{r}{\wt{r}} \leq \frac{d}{d-2}$.
		\end{itemize}
	\end{lem}
	
	\begin{lem}[Derivatives of differences, \cite{2011-Killip-Visan}]\label{2: Lemma-H^s,0<s<1}
		Set $F(z)=|z|^pz$, with $p >0$. Let $s\in (0,1)$ and $ 1<r_1, r_2, r_3 < \infty$ satisfy $\frac{1}{r_1} = \frac{1}{r_2} + \frac{1}{r_3}$. Then we have 
		\begin{align*}
			\||\nabla|^s(F(u)-F(v))\|_{L^{r_1}} \lesssim \|u\|_{L^{pr_2}}^p \||\nabla|^s (u-v)\|_{L^{r_3}}
			+ \||\nabla|^sv\|_{L^{r_3}}\|u-v\|^p_{L^{pr_2}}.
		\end{align*}
	\end{lem}
	
	\begin{lem}[Radial Strichartz estimates, \cite{2014-Guo-Wang,2012-Ke}]\label{RSE:radial strichartz estimate}
		Let $d\geq 3$, $\gamma \in \mathbb{R}$, $I\subset\mathbb{R}$, and let $u_0$, $F(t)$ be radial functions for all $t\in I$. Suppose that $(q,r)$ and $(\wt{q}, \wt{r})$ are radial-admissible pairs and satisfy the ``gap'' condition
		\begin{align*}
			\frac{2}{q}+\frac{d}{r}=\frac{d}{2}-\gamma; \quad \frac{2}{\wt{q}}+\frac{d}{\wt{r}}=\frac{d}{2}+\gamma.
		\end{align*}
		Then, for $a\in I$, we have
		\begin{align*}
			\|S(t)u_0\|_{L_t^{\infty}\dot{H}_x^{\gamma}\cap L_t^{q}L_x^{r}(I\times \mathbb{R}^d)} &\lesssim \|u_0\|_{\dot{H}^{\gamma}(\mathbb{R}^d)}, \\
			\left\| \int_{a}^{t}S(t-\tau)F(\tau)\mathrm{d}\tau \right\|_{L_t^{\infty}\dot{H}_x^{\gamma}\cap L_t^{q}L_x^{r}(I\times \mathbb{R}^d)} &\lesssim \|F\|_{L_t^{\wt{q}^{\prime}}L_x^{\wt{r}^{\prime}}(I\times \mathbb{R}^d)}.
		\end{align*} 
	\end{lem}
	
	\vspace{2cm}
	
	\section{Local results}
	
	\vspace{0.5cm}

	\subsection{Local well-posedness}\label{section-lwp}
	Now, we give the proof of Theorem \ref{thm:local-s-order-initialdata} (1). First, note that the critical case $p=\frac{4}{d+2s}$ has already been proved by Masaki \cite{2015-Masaki}; hence, it suffices to prove the local well-posedness in the subcritical range  $\frac{4s}{d+2s}<p< \frac{4}{d+2s}$. The precise argument is given as follows:
	\begin{prop}[Local well-posedness in the subcritical case]\label{prop:lwp-subcritical}
		Assume that $1\le d\le6$, $0<s<\min \{1, \frac{d}{2}\}$, and $u_0\in\F\dot H_x^s(\R^d)$. Let $\frac{4s}{d+2s}<p < \frac{4}{d+2s}$. Then there exists some $T = T(\norm{u_0}_{\F\dot H_x^s(\R^d)})> 0$ such that \eqref{NLS} admits a unique solution $u$ with $e^{-\frac12it\De}u\in C([0,T];\F\dot H_x^s(\R^d))$.
	\end{prop}
	We apply the pseudo conformal transform $\mathcal{T}$ to \eqref{NLS}, and then $\mathcal{U}$ satisfies \eqref{CNLS-I}
	%	\begin{align}\label{d-CNLS}
		%		i \pd_t \mathcal{U}+\frac{1}{2}\Delta \mathcal{U}=t^{\frac{dp}{2}-2}|\mathcal{U}|^p \mathcal{U},
		%	\end{align}
	with the asymptotic condition 
	\begin{align*}
		\lim_{t\rightarrow +\infty}\|S(-t)\mathcal{U}(t)- \mathcal{U}_+ \|_{\dot{H}_x^s(\mathbb{R}^d)}=0,
	\end{align*}
	where $$\mathcal{U}_+:=\mathcal{F}^{-1}\wb{u}_0\in \dot{H}_x^s(\mathbb{R}^d).$$
	Then, we reduce Proposition \ref{prop:lwp-subcritical} to the related result after applying the pseudo conformal transform:
	\begin{prop}\label{prop:lwp-subcritical-pseudo conformal transform}
		Let the assumptions in Proposition \ref{prop:lwp-subcritical} hold and $\mathcal{U}_{+}\in \dot{H}_x^s(\mathbb{R}^d)$. Then there exists some $T_0=T_0(\|\mathcal{U}_+\|_{\dot{H}_x^s(\R^d)})>0$ such that \eqref{CNLS-I} admits a unique solution $\mathcal{U}$ with $\mathcal{U} \in C([T_0,\infty);\dot{H}_x^s(\R^d))$.
	\end{prop}
	
	Note that $\frac{4s}{d+2s}<1$, and $\frac{4}{d+2s}$ may be larger than $1$. We consider two cases: $p<1$ and $p\geq 1$. When $p\geq 1$, we are able to obtain the local well-posedness by the standard contraction mapping argument. However, the nonlinearity fails to be $C^2$ continuous when $0<p<1$, which leads to an extra difficulty. Therefore, we need some auxiliary spacetime exponents when $p<1$:
	\begin{defn}\label{defn:lwp-parameters}
		Let $0<s<\min\{1, \frac{d}{2}\}$ and $\frac{4s}{d+2s}<p < \min \{ 1, \frac{4}{d-2s} \}$. 
		\begin{enumerate}
			\item When $p\leq \frac{2}{d-2s}$, we define $\theta:=\frac{1}{2} sp +\frac{1}{2}\max \{s-\frac{(d-2s)p}{4}, 0\}$, and $(q_0, r_0)$, $(\wt{q}, \wt{r})$ are given by
			\begin{equation}\label{eq:parameters-small}
				\begin{aligned}
					\frac{1}{q_0}&=\frac{(d-2s)p+4\theta-4s}{4(p+2)}, \quad \frac{1}{r_0}=\frac{d+\theta p}{d(p+2)}, \\
					\frac{1}{\wt{q}}&=\frac{(d+2s-4\theta)p+4(s-\theta)}{4(p+2)}, \quad \frac{1}{\wt{r}}=\frac{d+\theta p}{d(p+2)}.
				\end{aligned}
			\end{equation}
			\item 
			When $p>\frac{2}{d-2s}$, we define $\theta:=\frac{1}{2}sp+\frac{1}{2}\max \{s-\frac{1}{2}, 0\}$, and $(q_0, r_0)$, $(\wt{q}, \wt{r})$ are given by
			\begin{equation}\label{eq:parameters-large}
				\begin{aligned}
					\frac{1}{q_0}&=\frac{(d-2s)p-2(s-\theta)-1}{4(p+1)}, \quad \frac{1}{r_0}=\frac{d+2\theta p+1}{2d(p+1)}, \\
					\frac{1}{\wt{q}}&=\frac{1+2(s-\theta)}{4}, \quad \frac{1}{\wt{r}}=\frac{d-1}{2d}.
				\end{aligned}
			\end{equation}
			\item 
			Let $\theta$, $(q_0, r_0)$, and $(\wt{q}, \wt{r})$ be given above. We denote
			\begin{align}\label{defn:x1y1}
				\|f\|_{X_1(I)}:=\||\nabla|^{\theta}f\|_{L_t^{q_0}L_x^{r_0}(I\times\mathbb{R}^d)},\quad \|f\|_{Y_1(I)}:=\||\nabla|^{\theta}f\|_{L_t^{\wt{q}^{\prime}}L_x^{\wt{r}^{\prime}}(I\times\mathbb{R}^d)}.
			\end{align}
			\item 
			Let $\wt{r}_0$ satisfy
			\begin{align}
				\frac{2}{q_0}+\frac{d}{\wt{r}_0}=\frac{d}{2}.
			\end{align}
		\end{enumerate}
	\end{defn}
	\begin{remark}
		Now, we show that the pairs $(q_0, r_0)$ and $(\wt{q}, \wt{r})$ defined above are acceptable in the sense of Definition \ref{defn:acceptable-admissible}, and satisfy $2 < q_0,\wt{q} < \infty$, $0<\th<sp$, and
		\EQn{\label{eq:lwp-parameters-scaling}
			\frac{2}{q_0}+\frac{d}{r_0}=\frac{d}{2}-(s-\theta); \quad \frac{2}{\wt{q}}+\frac{d}{\wt{r}}=\frac{d}{2}+(s-\theta).
		} 
		First note that \eqref{eq:lwp-parameters-scaling} follows by direct computation. Moreover, we also show that $q_0, \wt{q}>4$ when $d=1$, which implies the pair $(q_0, \wt{r}_0)$ is admissible.
		\begin{enumerate}
			\item
			We consider the case when $p\le \frac{2}{d-2s}$. We observe that $s-\frac{(d-2s)p}{4}< sp$ is equivalent to $\frac{4s}{d+2s}<p$. Then, by the choice of $\th$, we have that $s-\frac{(d-2s)p}{4}<\theta<sp$. It is easy to see that $\frac{1}{q_0}>0$. Then \eqref{eq:lwp-parameters-scaling} yields
			\EQ{
				\frac{1}{q_0} + \frac{d}{r_0} = \frac{d}{2} - (s-\th) -\frac{1}{q_0} <\frac d2.
			}
			Thus, $(q_0, r_0)$ is acceptable.
			Furthermore, since $(d-2s)p\leq2$ and $\theta<s$, we have
			$0<(d-2s)p+4\theta-4s<2(p+2)$, and hence $2<q_0<\infty$. We also observe that  $\frac{1}{\wt q}<\frac{d}{2}-\frac{d}{\wt r}$ is equivalent to $s-\frac{(d-2s)p}{4}<\theta$. This shows that $(\wt{q}, \wt{r})$ is acceptable. Note that $s-\theta <\frac{(d-2s)p}{4}\leq \frac{1}{2}$. Then, we have that $(d+2s-4\theta)p+4(s-\theta)= (d-2s)p+4(s-\theta)(p+1)< 2+2(p+1)=2(p+2)$. This implies $2< \wt{q}<\infty$.
			\item 
			We next consider the case when $p> \frac{2}{d-2s}$. By a direct calculation, we have that
			$1-\frac{1}{2s}<\frac{4s}{d+2s}$ is equivalent to
			\EQ{
				4s^2-(2d-2)s +d>0.
			}
			We can check that this inequality holds for $1\le d\le 6$ and $\frac{1}{2}<s<1$. Therefore, we have that $s-\frac12<\frac{4s^2}{d+2s}<sp$, and then the choice of $\th$ is well-defined. Thus, we have that $s-\frac{1}{2}<\theta<sp$, and $2< \wt{q}<\infty$. 
			
			In order to obtain $2 < q_0< \infty$, it suffices to check that $0<(d-2s)p-2(s-\theta)-1 <2(p+1)$. $(d-2s)p-2(s-\th)-1>0$ follows from $p>\frac{2}{d-2s}$ and $s-\th<\frac12$. Moreover, by $p\leq \frac{4}{d-2s}$, for any $1\le d\leq 6$ and $0<s<1$, we have that $(d-2s-2)p\leq 3$, which implies $(d-2s)p-2(s-\theta)-1 <2(p+1)$. This shows that $2 < q_0< \infty$. 
			
			By a direct calculation, we can reduce the condition $\frac{d-2}{d}\leq \frac{r_0}{\wt{r}}\leq \frac{d}{d-2}$ to
			\EQ{
				(d^2-3d+2)p \le 2\th dp +4d-2;\quad (2\th(d-2)-(d^2-d))p\le2.
			}
			The first inequality follows from $p<1$ and
			$(d-1)(d-2)\leq4d-2$, while the second follows from $\theta<1$.
			\item We finally show that $q_0, \wt{q}>4$ when $d=1$. It is obvious that $q_0$ and $\wt{q}$ are given by \eqref{eq:parameters-small}. It follows from $p<1$ and $\theta<s$ that $(1-2s)p+4\theta-4s<p$, which implies $q_0>4$. We observe that $\wt{q}>4$ is equivalent to $(s-2\theta)p+2(s-\theta)<1$. It is a consequence of $s-\theta<\frac{(1-2s)p}{4}<\frac{1}{4}$ and $sp<\frac{p}{2}<\frac{1}{2}$.
		\end{enumerate}
	\end{remark}
	
	\begin{comment}
		\begin{remark}
			In order to ensure $\wt{q}\geq 2$ and $(\wt{q}, \wt{r})$ acceptable, we need the condition $d\leq 6$. However, if we only consider $\wt{q}\geq 1$, we can remove the restriction of $d$, see Remark \ref{remark:lwp-d<=6}.
		\end{remark}
	\end{comment}
	
	Next, we derive some estimates that will help us control the nonlinearity. 
	\begin{lem}[Nonlinear estimates]\label{Lemma: Nonlinear estimate}
		Let $\frac{4s}{d+2s}<p< \min \{1, \frac{4}{d+2s}\}$, $F(z):=|z|^pz$, and $I=[T, \infty )$. Assume that $\theta$, $(q_0, r_0)$, $(\wt{q}, \wt{r})$, $\wt{r}_0$, $X_1$, and $Y_1$ are given in Definition \ref{defn:lwp-parameters}. Then,
		\begin{align}\label{Nonlinear estimate -I}
			\left\|\int_{\infty}^{t}S(t-\tau)(\tau^{\frac{dp}{2}-2}F(\mathcal{U}(\ta)))\mathrm{d} \tau \right\|_{X_1(I)} \lesssim T^{\frac{(d+2s)p-4}{4}}\|\mathcal{U}\|^{p+1}_{X_1(I)},
		\end{align}
		and 
		\begin{equation}\label{Nonlinear estimate -II}
			\begin{aligned}
				&\left\|\int_{\infty}^{t}S(t-\tau)(\tau^{\frac{dp}{2}-2} F(\mathcal{U}(\ta)) - \tau^{\frac{dp}{2}-2}F(\mathcal{V}(\ta)) ) \mathrm{d} \tau \right\|_{X_1(I)} \\
				&\qquad \lesssim T^{\frac{(d+2s)p-4}{4}} \big( \||\nabla|^s \mathcal{U}\|^p_{L_t^{q_0}L_x^{\wt{r}_0}} + \||\nabla|^s(\mathcal{U}-\mathcal{V}) \|^p_{L_t^{q_0}L_x^{\wt{r}_0}}\big)\|\mathcal{U}-\mathcal{V}\|_{X_1(I)}.
			\end{aligned}
		\end{equation}
		
	\end{lem}
	\begin{proof}
		By the inhomogeneous Strichartz estimate in Lemma \ref{2: Lemma-inhomogeneous Strichartz estimates}, we have that
		\begin{align}\label{lwp-s-inhomogeneous-X1-Y1}
			\left\|\int_{\infty}^{t}S(t-\tau)f(\tau, \cdot )\mathrm{d} \tau \right\|_{X_1(I)} \lesssim \|f\|_{Y_1(I)}.
		\end{align}
		Let $q_1$ and $r_1$ be defined by
		\begin{align}\label{defn:q1r1}
			\frac{1}{q_1}=\frac{4-(d-2s)p}{4}; \quad \frac{1}{r_1}=\frac{1}{r_0}-\frac{\theta}{d}.
		\end{align}
		A direct calculation shows that 
		\begin{align*}
			1-\frac{1}{\wt{q}}=\frac{1}{q_1}+\frac{p+1}{q_0}; \quad 1-\frac{1}{\wt{r}} =\frac{p}{r_1}+\frac{1}{r_0}.
		\end{align*}
		In the following, we restrict the spacetime variable $(t,x)\in I\times\mathbb{R}^d$. By \eqref{lwp-s-inhomogeneous-X1-Y1}, Lemma \ref{2: Lemma-fractional chain rule 2}, H\"older's and Sobolev's inequalities,
		\begin{align*}
			\left\|\int_{\infty}^{t}|\nabla|^{\theta}S(t-\tau)(\tau^{\frac{dp}{2}-2}F(\mathcal{U}))\mathrm{d} \tau \right\|_{L_t^{q_0}L_x^{r_0}} & \lesssim \|t^{\frac{dp}{2}-2}|\nabla|^{\theta}F(\mathcal{U})\|_{L_t^{\wt{q}^{\prime}}L_x^{\wt{r}^{\prime}}}\\
			&\lesssim \|t^{\frac{dp}{2}-2}\|_{L_t^{q_1}}\|\mathcal{U}\|^p_{L_t^{q_0}L_x^{r_1}} \||\nabla|^{\theta}\mathcal{U}\|_{L_t^{q_0}L_x^{r_0}}\\
			&\lesssim T^{\frac{(d+2s)p-4}{4}}\|\mathcal{U}\|^{p+1}_{X_1(I)}.
		\end{align*}
		This gives \eqref{Nonlinear estimate -I}. 
		
		We next prove \eqref{Nonlinear estimate -II}. Note that
		\begin{align} \label{eq:FzFw}
			F(z)-F(w)=(z-w)\int_{0}^{1}F_z(w+s(z-w))\mathrm{d}s+(\bar{z}-\bar{w})\int_{0}^{1}F_{\bar{z}}(w+s(z-w))\mathrm{d}s,
		\end{align}
		where $F_z:=\frac{p+2}{2}|z|^p$ and $F_{\bar{z}}:=\frac{p}{2}z^{\frac{p+2}{2}}\bar{z}^{\frac{p-2}{2}}$, which are both H\"older continuous of order $p$. Therefore, by \eqref{lwp-s-inhomogeneous-X1-Y1} and \eqref{eq:FzFw}, we can reduce the proof of \eqref{Nonlinear estimate -II} to the following inequality
		\begin{equation}\label{Lemma: Nonlinear estimate II 3}
			\begin{aligned}
				&\big\|t^{\frac{dp}{2}-2}|\nabla|^{\theta}(F_z(u+v)w)\big\|_{L_t^{\wt{q}^{\prime}}L_x^{\wt{r}^{\prime}}}\\
				&\lesssim T^{\frac{(d+2s)p-4}{4}}\big(\||\nabla|^{s}u\|^p_{L_t^{q_0}L_x^{\wt{r}_0}} + \||\nabla|^{s}v\|^p_{L_t^{q_0}L_x^{\wt{r}_0}}\big)\||\nabla|^{\theta}w\|_{L_t^{q_0}L_x^{r_0}}.
			\end{aligned}
		\end{equation}
		
		Now, it suffices to prove \eqref{Lemma: Nonlinear estimate II 3}. Take $q_2$, $r_2$, and $r_3$ such that
		\begin{align}\label{defn:q2r2r3}
			\frac{1}{q_2}=\frac{p}{q_0}; \quad \frac{1}{r_2}=\frac{p}{r_1}+\frac{\theta}{d}; \quad \frac{1}{r_3}=\frac{1}{r_2}-\frac{\theta}{s\wt{r}_0}=(p- \frac{\theta}{s})\frac{1}{r_1}.
		\end{align}
		Then, by Lemma \ref{2: Lemma-fractional chain rule 1},
		\EQn{\label{Lemma: Nonlinear estimate II 3-1}
			\big\|t^{\frac{dp}{2}-2}|\nabla|^{\theta}(F_z(u+v)w)\big\|_{L_t^{\wt{q}^{\prime}}L_x^{\wt{r}^{\prime}}}
			& \lesssim \|t^{\frac{dp}{2}-2}\|_{L_t^{q_1}} \||\nabla|^{\theta}F_z(u+v)\|_{L_t^{q_2}L_x^{r_2}}\|w\|_{L_t^{q_0}L_x^{r_1}}\\
			& \qquad +\|t^{\frac{dp}{2}-2}\|_{L_t^{q_1}} \|u+v\|^p_{L_t^{q_0}L_x^{r_1}}\||\nabla|^{\theta}w\|_{L_t^{q_0}L_x^{r_0}}.
		}
		By Lemma \ref{2: Lemma-Holder continuous}, H\"older's and Sobolev's inequalities,
		\EQn{\label{Lemma: Nonlinear estimate II 3-2}
			\||\nabla|^{\theta}F_z(u+v)\|_{L_t^{q_2}L_x^{r_2}}
			&\lesssim \big\|\||u+v|^{p-\frac{\theta}{s}} \|_{L_x^{r_3}} \||\nabla|^{s}(u+v) \|^{\frac{\theta}{s}}_{L_x^{\wt{r}_0}} \big\|_{L_t^{q_2}}\\
			&\lesssim \|u+v\|^{p-\frac{\theta}{s}}_{L_t^{q_0}L_x^{r_1}} \| |\nabla|^{s}(u+v) \|^{\frac{\theta}{s}}_{L_t^{q_0}L_x^{\wt{r}_0}} \\
			&\lesssim \||\nabla|^{s}u\|^p_{L_t^{q_0}L_x^{\wt{r}_0}} + \||\nabla|^{s}v\|^p_{L_t^{q_0}L_x^{\wt{r}_0}},
		}
		Therefore, \eqref{Lemma: Nonlinear estimate II 3} follows from \eqref{Lemma: Nonlinear estimate II 3-1} and \eqref{Lemma: Nonlinear estimate II 3-2}.
	\end{proof}
	We are now in a position to prove Proposition \ref{prop:lwp-subcritical-pseudo conformal transform}.
	\begin{proof}[Proof of Proposition \ref{prop:lwp-subcritical-pseudo conformal transform}]
		We denote $I:=[T_0,\infty)$. Define 
		\begin{align*}
			\Phi(\mathcal{U}):=S(t)\mathcal{U}_{+}-i\int_{\infty}^{t}S(t-\tau)(\tau^{\frac{dp}{2}-2}|\mathcal{U}|^p\mathcal{U})\mathrm{d}\tau.
		\end{align*}
		If $p\geq 1$, we set $q_0=\wt{q} = \frac{4(p+2)}{(d-2s)p}$. If $p<1$, $(q_0, r_0)$ and $(\wt{q}, \wt{r})$ are defined in Definition \ref{defn:lwp-parameters}. We choose $\wt{r}_0$ and $\wt{r}_1$ such that 
		\begin{align*}
			\frac{2}{q_0}+\frac{d}{\wt{r}_0}=\frac{d}{2}, \quad \frac{2}{\wt{q}}+\frac{d}{\wt{r}_1}=\frac{d}{2}.
		\end{align*}
		Since $\wt{q} > \max\{2,\frac{4}{d}\}$, the pair $(\wt{q}, \wt{r}_1)$ is admissible. In the following, we only prove the case when $\frac{4}{d+2s}>1$, and the method is  similar when $\frac{4}{d+2s}\le1$.  
		
		For $1\leq p <\frac{4}{d+2s}$, we consider the resolution space
		\begin{align*}
			E_1(R):=\{\mathcal{U}\in C(I;\dot{H}_x^s(\mathbb{R}^d)): \||\nabla|^s \mathcal{U}\|_{ L_t^{\infty}L_x^{2}\cap L_t^{q_0}L_x^{\wt{r}_0}(I\times\R^d)}\leqslant 2R \},
		\end{align*}
		under the metric 
		\begin{align*}
			d_1(\mathcal{U},\mathcal{V}) := \||\nabla|^s( \mathcal{U}-\mathcal{V} )\|_{L_t^{\infty}L_x^{2}\cap L_t^{q_0}L_x^{\wt{r}_0}(I\times\R^d)}.
		\end{align*}
		In the following, we restrict the spacetime variable $(t,x) \in I\times\mathbb{R}^d$. By the Strichartz estimate, Lemma \ref{2: Lemma-fractional chain rule 2}, H\"older's and Sobolev's inequalities,
		\begin{equation}\label{proof of lwp s-order-1}
			\begin{aligned}
				\||\nabla|^s \Phi(\mathcal{U})\|_{L_t^{\infty}L_x^{2}\cap L_t^{q_0}L_x^{\wt{r}_0}} &\lesssim \|\mathcal{U}_+\|_{\dot{H}_x^s(\R^d)} + \| t^{\frac{dp}{2}-2} |\nabla|^s (|\mathcal{U}|^p \mathcal{U}) \|_{L_t^{\wt{q}^{\prime}}L_x^{\wt{r}_1^{\prime}}}\\
				&\lesssim \|\mathcal{U}_+\|_{\dot{H}_x^s(\R^d)} + \|t^{\frac{dp}{2}-2}\|_{L_t^{q_1}} \||\nabla|^{s}\mathcal{U}\|^{p+1}_{L_t^{q_0}L_x^{\wt{r}_0}}\\
				&\lesssim \|\mathcal{U}_+\|_{\dot{H}_x^s(\R^d)} + T_0^{\frac{(d+2s)p-4}{4}} \||\nabla|^{s}\mathcal{U}\|^{p+1}_{L_t^{q_0}L_x^{\wt{r}_0}}.
			\end{aligned}
		\end{equation}
		By Lemma \ref{2: Lemma-H^s,0<s<1} and modifying the above nonlinear estimates,
		\begin{align*}
			&\||\nabla|^s(\Phi(\mathcal{U})-\Phi(\mathcal{V}))\|_{L_t^{\infty}L_x^{2}\cap L_t^{q_0}L_x^{\wt{r}_0}}\\
			&\lesssim T_0^{\frac{(d+2s)p-4}{4}} \big( \||\nabla|^s\mathcal{U}\|^p_{L_t^{q_0}L_x^{\wt{r}_0}} \||\nabla|^s(\mathcal{U} -\mathcal{V})\|_{L_t^{q_0}L_x^{\wt{r}_0}} \\
			&\qquad \qquad\qquad + \||\nabla|^s \mathcal{V}\|_{L_t^{q_0}L_x^{\wt{r}_0}} \||\nabla|^s(\mathcal{U}-\mathcal{V})\|^p_{L_t^{q_0}L_x^{\wt{r}_0}} \big) \\
			&\lesssim T_0^{\frac{(d+2s)p-4}{4}} (\||\nabla|^s\mathcal{U}\|^p_{L_t^{q_0}L_x^{\wt{r}_0}} + \||\nabla|^s\mathcal{V} \|^p_{L_t^{q_0}L_x^{\wt{r}_0}}) \||\nabla|^s(\mathcal{U}-\mathcal{V})\|_{L_t^{q_0}L_x^{\wt{r}_0}}.
		\end{align*}
		Let $R:=C\|\mathcal{U}_+\|_{\dot{H}_x^s(\R^d)}$. Then we have that 
		\begin{align}\label{eq:lwp-casea-1-contraction-1}
			\||\nabla|^s \Phi (\mathcal{U}) \|_{L_t^{\infty}L_x^{2}\cap L_t^{q_0}L_x^{\wt{r}_0}}&\leq R + C T_0^{\frac{(d+2s)p-4}{4}} R^{p+1},
		\end{align}
		and
		\begin{align}\label{eq:lwp-casea-1-contraction-2}
			d_1(\Phi (\mathcal{U}), \Phi (\mathcal{V}))& \leq C T_0^{\frac{(d+2s)p-4}{4}} R^p d_1(\mathcal{U}, \mathcal{V}).
		\end{align}

		For $\frac{4s}{d+2s} < p < 1$, we consider the resolution space
		\begin{align*}
			E_2(R):=\{|\nabla|^s\mathcal{U}\in L_t^{q_0}L_x^{\wt{r}_0}(I\times \mathbb{R}^d): \||\nabla|^s \mathcal{U}\|_{L_t^{q_0}L_x^{\wt{r}_0}(I\times \mathbb{R}^d)}\leqslant 2R \},
		\end{align*}
		under the metric
		\begin{align*}
			d_2(\mathcal{U},\mathcal{V}):=\|\mathcal{U}-\mathcal{V}\|_{X_1(I)}= \||\nabla|^{\theta}( \mathcal{U}-\mathcal{V} )\|_{L_t^{q_0}L_x^{r_0}(I\times\mathbb{R}^d)}.
		\end{align*}
		Thus, $E_2(R)$ is a complete metric space. By the same argument as in \eqref{proof of lwp s-order-1},
		\begin{equation}\label{proof of lwp s-order-2}
			\begin{aligned}
				\||\nabla|^s \Phi(\mathcal{U})\|_{L_t^{\infty} L_x^{2} \cap L_t^{q_0}L_x^{\tilde{r}_0}}
				&\lesssim \|\mathcal{U}_+\|_{\dot{H}_x^s(\R^d)} + T_0^{\frac{(d+2s)p-4}{4}} \||\nabla|^{s}\mathcal{U}\|^{p+1}_{L_t^{q_0}L_x^{\wt{r}_0}}.
			\end{aligned}
		\end{equation}
		Let $R:= C\|\mathcal{U}_+\|_{\dot{H}_x^s(\R^d)}$. Then
		\begin{align}\label{eq:lwp-casea-2-contraction-1}
			\||\nabla|^s \Phi(\mathcal{U})\|_{L_t^{\infty} L_x^{2} \cap L_t^{q_0}L_x^{\tilde{r}_0}} \leq R+ C T_0^{\frac{(d+2s)p-4}{4}} R^{p+1}.
		\end{align}
		By Lemma \ref{Lemma: Nonlinear estimate}, 
		\begin{align}\label{eq:lwp-casea-2-contraction-2}
			d_2(\Phi(\mathcal{U}), \Phi(\mathcal{V})) \leq C T_0^{\frac{(d+2s)p-4}{4}} R^p d_2(\mathcal{U}, \mathcal{V}).
		\end{align}
		
		A simple calculation shows that $\frac{(d+2s)p-4}{4}< 0$. By \eqref{eq:lwp-casea-1-contraction-1}, \eqref{eq:lwp-casea-1-contraction-2}, \eqref{eq:lwp-casea-2-contraction-1}, and \eqref{eq:lwp-casea-2-contraction-2}, we are able to choose $T_0=T_0(\norm{\U_+}_{\dot H_x^s(\R^d)})>0$ sufficiently large such that the map $\Phi : E_j(R) \rightarrow E_j(R)$ $(j=1,2)$ is a contraction. It follows from the fixed-point theorem that the equation \eqref{CNLS-I} admits a local solution. Uniqueness follows from an argument similar to those used to prove \eqref{eq:lwp-casea-1-contraction-2} and \eqref{eq:lwp-casea-2-contraction-2}.
	\end{proof}
	
	\begin{remark}\label{remark:lwp-d<=6}
		In \cite{2015-Masaki}, Masaki proved the equation \eqref{NLS} is locally well posed when $p=\frac{4}{d+2s}$ by using the Lorentz type inhomogeneous Strichartz estimate and the stability argument from \cite{2005-Tao-Visan}. %They establish a short-time perturbations result, and combined with the local well posedness of inhomogeneous weight space $\mathcal{F}H^s$ to obtain the local well posedness of $\mathcal{F}\dot{H}^s$. 
		In our case, $p<\frac{4}{d+2s}$; we may follow their argument to remove the condition $d\leq 6$. We do not pursue this issue here and just give some references \cite{2011-Li-Zhang,2013-Killip-Visan}.
	\end{remark}

	\subsection{Ill-posedness}\label{section-illposed}
	Now, we give the ill-posedness in Theorem \ref{thm:local-s-order-initialdata} (2) in the following sense:
	
	\begin{prop}\label{ill-main theorem1}
		Let the assumptions in Theorem \ref{thm:local-s-order-initialdata} (2) hold. Then for any $0< \varepsilon ,\delta <1$, and for any sufficiently small $t>0$, there exist solutions $u_1$, $u_2$ of \eqref{NLS} with initial data $u_1(0),u_2(0)\in \mathcal{S}(\mathbb{R}^d)$ respectively, such that 
		\begin{align*}
			\|u_1(0)\|_{\mathcal{F}\dot{H}_x^s(\R^d)}+\|u_2(0)\|_{\mathcal{F}\dot{H}_x^s(\R^d)} &\leq C \varepsilon,\\
			\|u_1(0)-u_2(0)\|_{\mathcal{F}\dot{H}_x^s(\R^d)}&< C \delta,\\
			\|S(-t)u_1(t)-S(-t)u_2(t)\|_{\mathcal{F}\dot{H}_x^s(\R^d)}&>c\varepsilon.
		\end{align*}
		Thus, the solution operator fails to be uniformly continuous on $\mathcal{F}\dot{H}^s(\R^d)$.
	\end{prop}
	
	We apply the method in \cite{2003-Christ-Colliander-Tao}, relying on some quantitative analysis of the NLS equation
	\begin{align}\label{NLS: phi equation}
		\left\{
		\begin{aligned}
			&i\phi_t+\frac{\mu^2}{2}\Delta \phi=|\phi|^p \phi, \\
			&\phi(0,x)=\phi_0(x),
		\end{aligned}
		\right.
	\end{align}
	in the small dispersion regime $\mu \rightarrow 0$, where 
	\begin{align}\label{illposedness-phi transform}
		u(t,x)=\phi(t,\mu x),
	\end{align}
	and $u(t,x)$ is the solution of \eqref{NLS}. 
	
	Formally, as $\mu \rightarrow 0$ the equation \eqref{NLS: phi equation} approaches the ODE
	\begin{align}
		\left\{
		\begin{aligned}
			&i\phi_t=|\phi|^p \phi, \\
			&\phi(0,x)=\phi_0(x),
		\end{aligned}
		\right.
	\end{align}
	which has an explicit solution $\phi^{(0)}$ defined by
	\begin{align}
		\phi^{(0)}(t,x):=\phi_0(x)e^{-it|\phi_0(x)|^p}.
	\end{align}
	Moreover, one may choose a suitable function $\phi_0(x)$ such that for every $t\geq 0$, $\phi^{(0)}(t,x) \in \mathcal{S}(\R^d)$. For the ill-posedness construction below, we fix $w(x)=e^{-|x|^2}$ and take $\phi_0=aw$ with $a\in[\frac{1}{2},1]$; then $\phi^{(0)}(t)\in\mathcal S(\mathbb R^d)$ for every $t$.
	
	The following lemma is our main quantitative result.
	
	\begin{lem}\label{illposedness-Lem2}
		Let the assumptions in Proposition \ref{ill-main theorem1} hold. Then there exist constants $C,c>0$ depending on all the above parameters, such that if $0<\mu\leq c$ is a sufficiently small real number, then for $T=c|\log \mu|^{c}$ there exists a solution $\phi\in C([-T,T];L_x^2(\mathbb{R}^d))$ of \eqref{NLS: phi equation} satisfying 
		\begin{align*}
			\sup_{|t|\leqslant c|\log \mu|^{c}} \|\phi(t)-\phi^{(0)}(t)\|_{L_x^2(\mathbb{R}^d)}\leq C\mu.
		\end{align*}
	\end{lem}
	\begin{proof}
		We define 
		\begin{align*}
			F(z):=|z|^pz;\quad w:=\phi-\phi^{(0)};
		\end{align*}
		then $w$ is a solution of the Cauchy problem
		\begin{align}\label{Lem: Cauchy problem}
			\left\{
			\begin{aligned}
				&iw_t+\frac{\mu^2}{2}\Delta w=-\frac{\mu^2}{2}\Delta\phi^{(0)}+F(\phi^{(0)}+w)-F(\phi^{(0)}), \\
				&w(0,x)=0.
			\end{aligned}
			\right.
		\end{align}
		By the global well-posedness theory in the energy space \cite{1987-Kato,2007-Visan}, $w=\phi-\phi^{(0)}$ is globally defined. Thus, it suffices to show that 
		\begin{align}\label{Lem- bound}
			\sup_{|t|\leq T}\|w(t)\|_{L_x^{2}(\mathbb{R}^d)}\leq C\mu,
		\end{align} 
		where $0\leq T\leq c|\log\mu|^c$.
		
		By \eqref{Lem: Cauchy problem} and H\"older's inequality,
		\begin{equation}\label{Lem-derivate}
			\begin{aligned}
				&\frac{\mathrm{d}}{\mathrm{d}t}\|w(t)\|^2_{L_x^2(\mathbb{R}^d)}=2\re \int_{\mathbb{R}^d}w_t\bar{w}\mathrm{d}x\\
				&=-2 \im\int_{\mathbb{R}^d}\big(\frac{\mu^2}{2}\Delta w+\frac{\mu^2}{2}\Delta \phi^{(0)}-(F(\phi^{(0)}+w)-F(\phi^{(0)}))\big)\bar{w}\mathrm{d}x\\
				&\leq C \mu^2 \|w(t)\|_{L_x^2(\mathbb{R}^d)} \|\Delta\phi^{(0)}\|_{L_x^2(\mathbb{R}^d)} +C\left|\im\int_{\mathbb{R}^d}(F(\phi^{(0)}+w)-F(\phi^{(0)}))\bar{w}\mathrm{d}x\right|.
			\end{aligned}
		\end{equation}
		Denote
		\begin{align*}
			A(t):=\{ x\in\mathbb{R}^d: |w(t,x)|\leq |\phi^{(0)}(t,x)| \}; \quad A^{c}(t):=\mathbb{R}^d\backslash A(t).
		\end{align*}
		Then,
		\begin{equation}\label{Lem: Im F-1}
			\begin{aligned}
				&\left|\im\int_{\mathbb{R}^d}\chi_{A(t)}(F(\phi^{(0)}+w)-F(\phi^{(0)}))\bar{w}\mathrm{d}x\right| \\
				&\leq C \int_{\mathbb{R}^d}\chi_{A(t)}|w|(|\phi^{(0)}|^p+|w|^p)|\bar{w}|\mathrm{d}x\\
				&\leq C \|w(t)\|_{L_x^2(\mathbb{R}^d)}\|\chi_{A(t)}|\phi^{(0)}(t)|^pw(t)\|_{L_x^2(\mathbb{R}^d)},
			\end{aligned}
		\end{equation}
		and 
		\begin{equation}\label{Lem: Im F-2}
			\begin{aligned}
				&\left|\im\int_{\mathbb{R}^d}\chi_{A^c(t)}(F(\phi^{(0)}+w)-F(\phi^{(0)}))\bar{w}\mathrm{d}x\right|\\
				&=\left|\im\int_{\mathbb{R}^d}\chi_{A^c(t)}(|\phi^{(0)}+w|^p-|\phi^{(0)}|^p)\phi^{(0)}\bar{w}\mathrm{d}x\right|\\
				&\leq  \|w(t)\|_{L_x^2(\mathbb{R}^d)} \|\chi_{A^{c}(t)}|w(t)|^p\phi^{(0)}(t)\|_{L_x^2(\mathbb{R}^d)}.
			\end{aligned}
		\end{equation}
		Since $p \leq 1$ and $\phi_0$ is a Schwartz function, by \eqref{Lem: Im F-1} and \eqref{Lem: Im F-2} we have that
		\begin{align}\label{Lem-Im F}
			\left|\im\int_{\mathbb{R}^d}(F(\phi^{(0)}+w)-F(\phi^{(0)}))\bar{w}\mathrm{d}x\right|\leq C\|w(t)\|^2_{L_x^2(\mathbb{R}^d)}.
		\end{align}
		By \eqref{Lem-derivate} and \eqref{Lem-Im F},
		\begin{align*}
			\frac{\mathrm{d}}{\mathrm{d}t}\|w(t)\|_{L_x^2(\mathbb{R}^d)}\leq C\mu^2(1+|t|)^C+C\|w(t)\|_{L_x^2(\mathbb{R}^d)}.
		\end{align*}
		By Gronwall's inequality and the initial condition $w(0)=0$, 
		\begin{align*}
			\|w(t)\|_{L_x^2(\mathbb{R}^d)}\leq C\mu^2e^{C(1+|t|)^C}.
		\end{align*}
		Thus, if $|t|<c|\log\mu|^c$ for suitably chosen $c$ and $\mu$ is sufficiently small, we obtain \eqref{Lem- bound}.
	\end{proof}
	
	By Lemma \ref{illposedness-Lem2}, it follows that for $\mu\leq c$ there exists a solution $\phi^{(a,\mu)}(t,x)$ to the equation \eqref{NLS: phi equation} with initial data
	\begin{align*}
		\phi^{(a,\mu)}(0,x):=aw(x),
	\end{align*}
	where $a\in[\frac{1}{2},1]$ and $w(x)=e^{-|x|^2}\in \mathcal{S}(\mathbb{R}^d)$. Moreover, 
	\begin{equation}\label{illposedness-estimate12}
		\sup_{|t|\leq c|\log \mu|^{c}} \|\phi^{(a,\mu)}(t)-\phi^{(a,0)}(t)\|_{L_x^2(\mathbb{R}^d)}\leq C\mu,
	\end{equation}
	where 
	\begin{align}\label{ill-initial data}
		\phi^{(a,0)}(t,x)=aw(x)e^{-ita^p|w(x)|^p}.
	\end{align}
	
	By \eqref{illposedness-scale transform} and \eqref{illposedness-phi transform}, we obtain a three parameter family of solutions $u^{(a,\mu,\lambda)}$ to the NLS equation \eqref{NLS}, where $\frac{1}{2}\leq a\leq 1$, $0<\lambda < \mu \ll 1$, and 
	\begin{align*}
		u^{(a,\mu,\lambda)}(t,x)=\lambda^{-\frac{2}{p}} \phi^{(a,\mu)}(\lambda^{-2}t,\lambda^{-1}\mu x).
	\end{align*}
	We prove the ill-posedness result by considering the values of the three parameters.
	\begin{proof}[Proof of Proposition \ref{ill-main theorem1}]
		The initial data satisfy
		\begin{align*}
			u^{(a,\mu,\lambda)}(0,x)=\lambda^{-\frac{2}{p}} a w(\lambda^{-1}\mu x).
		\end{align*} 
		By a simple calculation,
		\begin{equation}\label{illpose-result-1}
			\begin{aligned}
				\|u^{(a,\mu,\lambda)}(0,x)\|_{\mathcal{F}\dot{H}_x^s(\mathbb{R}^d)}&=a\lambda^{\frac{d}{2}-\frac{2}{p}+s}\mu^{-\frac{d}{2}-s}\|w(x)\|_{\mathcal{F}\dot{H}_x^s(\mathbb{R}^d)},\\
				\|u^{(a,\mu,\lambda)}(0,x)-u^{(a^{\prime},\mu,\lambda)}(0,x)\|_{\mathcal{F}\dot{H}_x^s(\mathbb{R}^d)}&=|a-a^{\prime}|\lambda^{\frac{d}{2}-\frac{2}{p}+s}\mu^{-\frac{d}{2}-s}\|w(x)\|_{\mathcal{F}\dot{H}_x^s(\mathbb{R}^d)}.
			\end{aligned}
		\end{equation}
		Let $0<\varepsilon<1$ be fixed. Since $\frac{d}{2}-\frac{2}{p}+s>0$, we set
		\begin{align*}
			\lambda^{\frac{d}{2}-\frac{2}{p}+s}\mu^{-\frac{d}{2}-s}=\varepsilon.
		\end{align*}
		On the other hand,
		\begin{equation}\label{illposedness-dis-uniformal-continue}
			\begin{aligned}
				&\big\||x|^sS(-\lambda^2 t)\big( u^{(a,\mu,\lambda)}(\lambda^2 t ,y)-u^{(a^{\prime},\mu,\lambda)}(\lambda^2 t,y) \big)(x)\big\|_{L_x^2(\mathbb{R}^d)}\\
				&=\varepsilon\big\||x|^sS(-\mu^2t)\big(\phi^{(a,\mu)}(t,y)-\phi^{(a^{\prime},\mu)}(t,y)\big)(x)\big\|_{L_x^2(\mathbb{R}^d)}\\
				&\geq \varepsilon\big\|\chi_{|x|\geq 1} S(-\mu^2t)\big(\phi^{(a,\mu)} (t,y) -\phi^{(a^{\prime},\mu)}(t,y)\big)(x)\big\|_{L_x^2(\mathbb{R}^d)}\\
				&\geq \varepsilon\big\|\chi_{|x|\geq 1}S(-\mu^2 t)(\phi^{(a,0)} -\phi^{(a^{\prime},0)})\big\|_{L_x^2(\mathbb{R}^d)} \\
				&\qquad - \varepsilon\big\|\chi_{|x|\geq 1}S(-\mu^2 t)(\phi^{(a,\mu)}-\phi^{(a,0)})\big\|_{L_x^2(\mathbb{R}^d)}\\
				&\qquad -\varepsilon\big\|\chi_{|x|\geq 1}S(-\mu^2 t)(\phi^{(a^{\prime},\mu)}-\phi^{(a^{\prime},0)})\big\|_{L_x^2(\mathbb{R}^d)}.
			\end{aligned}
		\end{equation}
		From an inspection of \eqref{ill-initial data}, we see that there exists a time $T=T(a,a^{\prime})>0$ such that 
		\begin{align*}
			\|\chi_{|x|\geq 1} (\phi^{(a,0)}(T) - \phi^{(a^{\prime},0)}(T)) \|_{L_x^2(\mathbb{R}^d)}\geq c,
		\end{align*}
		where $c>0$ is independent of $a$ and $a^{\prime}$. Thus, by choosing $\mu$ sufficiently small, we have that
		\begin{align}\label{illpose-1}
			\big\|\chi_{|x|\geq 1}S(-\mu^2 T) (\phi^{(a,0)}(T) - \phi^{(a^{\prime},0)}(T))\big\|_{L_x^2(\mathbb{R}^d)}\geq \frac{c}{2}.
		\end{align}
		By \eqref{illposedness-estimate12},
		\begin{equation}\label{illpose-2}
			\begin{aligned}
				\big\|\chi_{|x|\geq 1}S(-\mu^2 T) (\phi^{(a,\mu)}(T) -\phi^{(a,0)}(T)) \big\|_{L_x^2(\mathbb{R}^d)}&\leq \|\phi^{(a,\mu)}(T) -\phi^{(a,0)}(T)\|_{L_x^2(\mathbb{R}^d)}\\
				&\leq C\mu.
			\end{aligned}
		\end{equation}
		It follows from \eqref{illposedness-dis-uniformal-continue}, \eqref{illpose-1}, and \eqref{illpose-2} that
		\begin{align}\label{illpose-result-2}
			\big\||x|^sS(-\lambda^2 T)\big( u^{(a,\mu,\lambda)}(\lambda^2 T )-u^{(a^{\prime},\mu,\lambda)}(\lambda^2 T) \big)\big\|_{L_x^2(\mathbb{R}^d)}>c\varepsilon.
		\end{align}
		Let $|a-a^{\prime}|<\delta$. Then $\mu\rightarrow 0$ forces the time $\lambda^2 T$ to approach zero as well. By \eqref{illpose-result-1} and \eqref{illpose-result-2}, we finish the proof of Proposition \ref{ill-main theorem1}.
	\end{proof}

	\subsection{Local well-posedness away from the origin}
	Similarly, applying the pseudo conformal transform $\mathcal{T}$ to \eqref{NLS-t_0}, we find that $\mathcal{U}=\mathcal{T}u$ satisfies the following equation:
	\begin{align}\label{CNLS-t_0}
		\left\{
		\begin{aligned}
			&i\pd_t\mathcal{U}+\frac{1}{2}\Delta \mathcal{U}=t^{\frac{dp}{2}-2}|\mathcal{U}|^p \mathcal{U}, \\
			&S(-t_0^{-1})\mathcal{U}(t_0^{-1})=\mathcal{F}^{-1}\overline{S(-t_0)u_1}.
		\end{aligned}
		\right.
	\end{align}

	\begin{proof}[Proof of Theorem \ref{thm:local-1order-awayfromorigin}]
		We denote $I:=[t_0^{-1}-T_0,t_0^{-1}+T_0]$ and $\mathcal{U}_+:=\mathcal{F}^{-1}\overline{S(-t_0)u_1}\in \dot{H}_x^s(\mathbb{R}^d)$. Set $0< T_0<\frac{1}{2t_0}$ temporarily.  The crucial observation is that $0,\infty \notin I$. If $p\geq 1$, we set $(q_0, \wt{r}_0)=\big( \frac{4(p+2)}{(d-2s)p}, \frac{d(p+2)}{d+sp}\big)$. If $p<1$, $(q_0,\wt r_0)$ is given in Definition \ref{defn:lwp-parameters}.
		By modifying \eqref{proof of lwp s-order-1} and \eqref{proof of lwp s-order-2}, we have that
		\begin{align*}
			\||\nabla|^s \mathcal{U}\|_{L_t^{\infty}L_x^2\cap L_t^{q_0}L_x^{\wt{r}_0}(I\times\R^d)}
			&\lesssim \|\mathcal{U}_+\|_{\dot{H}_x^s(\mathbb{R}^d)} + t_0^{2-\frac{dp}{2}} |I|^{\frac{1}{q_1}} \||\nabla|^{s}\mathcal{U}\|^{p+1}_{L_t^{q_0}L_x^{\wt{r}_0}(I\times\mathbb{R}^d)},
		\end{align*}
		where $\frac{1}{q_1}=\frac{4-(d-2s)p}{4}$.
		
		Note that $\frac{1}{q_1} > 0$. We are able to choose $T_0= T_0(\|\mathcal{U}_+\|_{\dot{H}_x^s(\R^d)})$  sufficiently small such that $|I|^{\frac{1}{q_1}}\ll 1$, and the rest of the proof is similar to Proposition \ref{prop:lwp-subcritical-pseudo conformal transform}. 
	\end{proof}
	\begin{remark}
		Similarly, we can also obtain a local solution when $p= \frac{4}{d-2s}$. However, the time $T_0$ will depend on the profile of $\mathcal{U}_+$, and not only on its norm. Indeed, if $p= \frac{4}{d-2s} $, then $\frac{1}{q_1} = 0$. Since $2<q_0<\infty$, we consider the space
		\begin{align*}
			E(\delta)= \{ |\nabla|^s\mathcal{U}\in L_t^{q_0}L_x^{\wt{r}_0}(I\times \mathbb{R}^d): \||\nabla|^s \mathcal{U}\|_{L_t^{q_0}L_x^{\wt{r}_0}(I\times\mathbb{R}^d)}\leq 2 \delta \}.
		\end{align*} 
		Choose $T_0=T_0(\mathcal{U}_+)$ sufficiently small such that $\delta =C \|S(t)|\nabla|^s \mathcal{U}_+\|_{L_t^{q_0}L_x^{\wt{r}_0}(I\times\mathbb{R}^d)} \ll 1$, and the rest of the proof is similar to Proposition \ref{prop:lwp-subcritical-pseudo conformal transform}.
	\end{remark}

	\vspace{2cm}
	
	\section{Proof of Theorem \ref{main-theorem1}}
	
	\vspace{0.5cm}

	We apply the transform $u(t,x)\rightarrow \overline{u(-t, x)}$, and then it suffices to prove Theorem \ref{main-theorem1} for $t>0$.

	\subsection{Local theory I: near the origin}
	We first study the local well-posedness  with initial data in $\mathcal{F}\dot{H}_x^{-s_c}\cap\mathcal{F}\dot{H}_x^1(\mathbb{R}^3)$ at the origin.
	\begin{prop}[Local well-posedness at $t=0$]\label{LWP1}
		Let the assumptions in Theorem \ref{main-theorem1} hold. Then, there exists  $I=[0,t_0]$ such that the equation \eqref{NLS} admits a unique solution $ S(-t)u\in C(I;\mathcal{F}\dot{H}^{-s_c}_x\cap\mathcal{F}\dot{H}_x^{1}(\mathbb{R}^3))$. Moreover, for any admissible pair $(q,r)$ with $q\neq\infty$, it holds that
		\begin{align*}
			\||J(t)|^{-s_c}u\|_{L_t^{\infty}L_x^2 \cap L_t^{q, 2}L_x^r(I\times\R^3)}+ \|J(t)u\|_{L_t^{\infty}L_x^2 \cap L_t^{q, 2}L_x^r(I\times\R^3)}  \leq C \|u_0\|_{\mathcal{F}\dot{H}^{-s_c}_x\cap\mathcal{F}\dot{H}_x^{1}(\R^3)}.
		\end{align*}
	\end{prop}
	
	\begin{proof}
		For all positive constants $R$ and $\delta$, we define $$E(R,\delta):=\{S(-t)u\in C(I;\mathcal{F} \dot{H}^{-s_c}_x \cap \mathcal{F} \dot{H}_x^{1}(\mathbb{R}^3)): \|u\|_{X_2(I)} \leq 2\delta, \|u\|_{Y_2(I)} \leq 2R \},$$
		under the metric
		\begin{align*}
			d(u, v) = \|u-v\|_{X_2(I)} + \|u-v\|_{Y_2(I)},
		\end{align*} 
		where
		\begin{align*}
			&\|u\|_{X_2(I)}:=\||J(t)|^{-s_c}u\|_{L_t^{4,2}L_x^3(I\times \mathbb{R}^3)}+\|J(t)u\|_{L_t^{4,2}L_x^3(I\times \mathbb{R}^3)};\\
			&\|u\|_{Y_2(I)}:=\||J(t)|^{-s_c}u\|_{L_t^{\infty}L_x^2(I\times\R^3)}+\|J(t)u\|_{L_t^{\infty}L_x^2(I\times\R^3)}.
		\end{align*}
		Define the solution map $$\Phi(u):=S(t)u_0-i\int_{0}^{t}S(t-\tau)(|u|^pu)\mathrm{d}\tau.$$
		By the Strichartz estimate  and $u_0\in \mathcal{F}\dot{H}^{-s_c}_x\cap\mathcal{F}\dot{H}_x^{1}(\mathbb{R}^3)$, for any $\delta>0$, there exists $t_0^{\prime}>0$ such that 
		\EQ{
			\|S(t)u_0\|_{X_2([0, t_0'])}&\leq \delta,
		}
		and
		\EQ{
			\|S(t)u_0\|_{Y_2([0,t'_0])}&\leq C\|u_0\|_{\mathcal{F}\dot{H}^{-s_c}_x\cap\mathcal{F}\dot{H}_x^{1}(\R^3)} .
		}
		Let 
		\begin{align*}
			R:=C\|u_0\|_{\mathcal{F}\dot{H}^{-s_c}_x\cap\mathcal{F}\dot{H}_x^{1}(\R^3)}.
		\end{align*}
		Set $t_0< t_0^{\prime}$ temporarily. In the following, we restrict the spacetime variable $(t,x) \in I\times\mathbb{R}^3$. Let $\wt{u}=M(-t)u$. Note that $|J(t)|^sS(t-\tau)=S(t-\tau)|J(\tau)|^s$. It follows from the Strichartz estimate, Lemma \ref{2: Lemma-fractional chain rule 2}, H\"older's and Sobolev's inequalities that for any admissible pair $(q,r)$ with $q\neq\infty$,
		\begin{equation}\label{estimate-J(t)-inhomogeneous}
			\begin{aligned}
				&\left\||J(t)|^{-s_c}\int_{0}^{t}S(t-\tau)(|u|^pu)(\tau)\mathrm{d}\tau\right\|_{L_t^{\infty}L_x^2 \cap L_t^{q, 2}L_x^r}\\
				&\lesssim \||J(t)|^{-s_c}(|u|^pu)\|_{L_t^{\frac{4}{9-5p},2}L_x^{\frac{6}{5p-2}}}\\
				&= \| |t\nabla|^{-s_c} (|\wt{u}|^p\wt{u}) \|_{L_t^{\frac{4}{9-5p},2} L_x^{\frac{6}{5p-2}}}\\
				&\lesssim \big\| \|\wt{u}\|^p_{L_x^{\frac{6p}{5p-4}}} \||J(t)|^{-s_c}u\|_{L_x^3} \big\|_{L_t^{\frac{4}{9-5p},2}}\\
				&\lesssim \|t^{s_cp}\|_{L_t^{-\frac{1}{s_cp},\infty}} \||J(t)|^{-s_c}u\|^p_{L_t^{4,\infty}L_x^3} \||J(t)|^{-s_c}u\|_{L_t^{4,2}L_x^3}\\
				&\lesssim \|u\|^{p+1}_{X_2(I)}.
			\end{aligned}
		\end{equation}
		Similarly,
		\begin{align*}
			\left\|J(t)\int_{0}^{t}S(t-\tau)(|u|^pu)(\tau)\mathrm{d}\tau\right\|_{L_t^{\infty}L_x^2 \cap L_t^{q, 2}L_x^r} \lesssim \|u\|^{p+1}_{X_2(I)}.
		\end{align*}
		Thus, we have that
		\begin{align*}
			\|\Phi(u)\|_{X_2(I)}\leq \delta+C\delta^{p+1},
		\end{align*}
		and
		\begin{align*}
			\|\Phi(u)\|_{Y_2(I)}&\leq R+C\delta^{p+1}.
		\end{align*}
		Let $\alpha=1$ or $-s_c$, by modifying the above nonlinear estimates, and by Lemma \ref{2: Lemma-H^s,0<s<1},
		\begin{equation}\label{estimate-J(t)(u-v)}
			\begin{aligned}
				\||J(t)|^{\alpha}(\Phi(u) - \Phi(v))\|_{L_t^{\infty}L_x^2 \cap L_t^{4, 2}L_x^3} &\lesssim \||J(t)|^{\alpha} (|u|^pu - |v|^pv) \|_{L_t^{\frac{4}{9-5p},2}L_x^{\frac{6}{5p-2}}}\\
				&\lesssim\big\| \|\wt{u}\|^p_{L_x^{\frac{6p}{5p-4}}} \||J(t)|^{\alpha} (u-v)\|_{L_x^3} \big\|_{L_t^{\frac{4}{9-5p},2}}\\
				&\qquad +\big\|\||J(t)|^{\alpha} v\|_{L_x^3}\|\wt{u}- \wt{v}\|^p_{L_x^{\frac{6p}{5p-4}}} \big\|_{L_t^{\frac{4}{9-5p},2}}\\
				%&\lesssim \||J(t)|^{-s_c} u\|^p_{L_t^{4,2}L_x^3} \||J(t)|^{\alpha} (u-v)\|_{L_t^{4,2}L_x^3} \\
				%&\qquad + \||J(t)|^{\alpha} v \|_{L_t^{4,2}L_x^3} \||J(t)|^{-s_c} (u-v) \|^p_{L_t^{4,2}L_x^3}\\
				&\lesssim \delta^p\|u-v\|_{X_2(I)}.
			\end{aligned}
		\end{equation}
		Therefore, 
		\begin{align*}
			d(\Phi(u), \Phi(v))\leq C \delta^p d(u,v).
		\end{align*}
		By choosing $\delta>0$ sufficiently small, we prove that $\Phi$ maps $E(R,\delta)$ into itself and is a contraction. This finishes the proof.
	\end{proof}
	\begin{remark}
		In the following, we will regard $t_0$ and $\norm{u_0}_{\mathcal{F}\dot{H}^{-s_c}_x\cap\mathcal{F}\dot{H}_x^{1}(\R^3)}$ as constants, and omit their dependence for simplicity. 
	\end{remark}
	
	We apply the pseudo conformal transform $\mathcal{T}$ to \eqref{NLS}, and then $\mathcal{U}=\TT u$ satisfies \eqref{CNLS-I}
	%	\begin{align}\label{CNLS}
		%		i \pd_t \mathcal{U}+\frac{1}{2}\Delta \mathcal{U}=t^{\frac{3p}{2}-2}|\mathcal{U}|^p \mathcal{U},
		%	\end{align}
	with asymptotic condition 
	\begin{align*}
		\lim_{t\rightarrow +\infty}\|S(-t)\mathcal{U}(t)- \mathcal{U}_+ \|_{\dot{H}_x^{-s_c}\cap \dot{H}_x^1(\mathbb{R}^3)}=0,
	\end{align*}
	where $\mathcal{U}_+:=\mathcal{F}^{-1}\bar{u}_0\in \dot{H}_x^{-s_c}\cap \dot{H}_x^1(\mathbb{R}^3)$.
	
	Repeating the contraction argument of Proposition \ref{LWP1} for the transformed Duhamel formula, using the Lorentz-space Strichartz estimates, yields the following corollary.
	\begin{cor}\label{cor:local-weighted-1}
		Let the assumptions in Theorem \ref{main-theorem1} hold. There exists $T_0>0$ such that the equation \eqref{CNLS-I} admits a unique solution $ \mathcal{U}\in C([T_0,\infty);\dot{H}^{-s_c}_x\cap\dot{H}_x^{1}(\mathbb{R}^3)). $ Moreover, for any admissible pair $(q,r)$ with $q\neq\infty$,
		\begin{align}\label{esti-bound for (T_0, infty)}
			\||\nabla|^{-s_c} \mathcal{U} \|_{L_t^{\infty}L_x^2 \cap L_t^{q, 2}L_x^r([T_0,\infty)\times\R^3)} + \|\nabla \mathcal{U} \|_{L_t^{\infty}L_x^2 \cap L_t^{q, 2}L_x^r([T_0,\infty)\times\R^3)}\lesssim 1.
		\end{align}
	\end{cor}
	Now, we reduce the problem to extending the solution in Corollary \ref{cor:local-weighted-1} from $T_0$ backward to $0$.

	\subsection{Local theory II: away from the origin}
	Since Proposition \ref{LWP1} is not sufficient to extend the solution globally, we need the following local result with initial time less than $ T_0$, only requiring the initial data in $ \dot H_x^1(\mathbb{R}^3)$. This result is also equivalent to the blow-up criterion by $\dot H_x^1(\mathbb{R}^3)$-norm.
	\begin{prop}[Local well-posedness away from $t=0$]\label{GWP-I:prop H1}
		Let the assumptions in Theorem \ref{main-theorem1} hold, and fix any $T>0$. Assume that there exists $M>0$ and some $t_1\in (T,T_0)$ such that
		\EQ{
			\norm{\mathcal{U}(t_1)}_{\dot{H}_x^1(\mathbb{R}^3)} \le M.
		}  
		Then, there exists some $0<t_2\leq \min\{1,\frac{T}{2}\}$ depending on $T$ and $M$ such that the equation \eqref{CNLS-I} admits a unique solution $ \mathcal{U}\in C([t_1-t_2,t_1];\dot{H}^1_x(\mathbb{R}^3)).$ 
	\end{prop}
	\begin{proof}
		Define the solution map
		\begin{align*}
			\Phi(\mathcal{U}):=S(t-t_1)\mathcal{U}(t_1)-i\int_{t_1}^{t}S(t-\tau)(\tau^{\frac{3p}{2}-2}|\mathcal{U}|^p\mathcal{U})\mathrm{d}\tau.
		\end{align*}
		Take $t_2$ that will be defined later, and set $0<t_2\leq \min\{1,\frac{T}{2}\}$ temporarily. Set $I:=[t_1-t_2,t_1]$. Then, $t_2\le \frac T2$ implies that $t_1-t_2>t_2>0$. By the Strichartz estimate,
		\begin{align*}
			\|\nabla S(t-t_1)\mathcal{U}(t_1)\|_{L_t^{\infty}L_x^2 \cap L_t^2L_x^6(I\times\R^3)} \leq  C\|\nabla\mathcal{U}(t_1)\|_{L_x^2(\mathbb{R}^3)}\le CM.
		\end{align*}
		We denote the right hand side of the above estimate as  $R:=CM$. Define the resolution space
		\begin{align*}
			E(R):=\{\mathcal{U}(t)\in C(I;\dot{H}^1_x):\|\nabla\mathcal{U}(t)\|_{L_t^{\infty}L_x^2 \cap L_t^2L_x^6(I\times\R^3)}\leq 2R\},
		\end{align*}
		with
		\begin{align*}
			d(\mathcal{U}, \mathcal{V}):=\|\nabla\mathcal{U} - \nabla\mathcal{V}\|_{L_t^{\infty}L_x^2 \cap L_t^2L_x^6(I\times\R^3)}.
		\end{align*}
		In the following, we restrict the spacetime variable $(t,x) \in I\times\mathbb{R}^3$. By the Strichartz estimate, H\"older's and Sobolev's inequalities,
		\begin{equation}\label{proof-Proposition-away from 0-1}
			\begin{aligned}
				\left\|\nabla\int_{t_1}^{t}S(t-\tau)(\tau^{\frac{3p}{2}-2}|\mathcal{U}|^p\mathcal{U})\mathrm{d}\tau\right\|_{L_t^{\infty}L_x^2 \cap L_t^2L_x^6}&\lesssim\|\nabla(t^{\frac{3p}{2}-2}|\mathcal{U}|^p\mathcal{U})\|_{L_t^{\frac{8}{7}}L_x^{\frac{12}{7}}}\\
				&\lesssim (t_1-t_2)^{\frac{3p}{2}-2}\big\|\|\nabla\mathcal{U}\|_{L_x^{\frac{12}{5}}}\|\mathcal{U}\|^p_{L_x^{6p}}\big\|_{L_t^{\frac{8}{7}}}\\
				&\lesssim t_2^{\frac{3p}{2}-2}|I|^{\frac{4-p}{4p}\cdot p}\|\nabla\mathcal{U}\|_{L_t^{8}L_x^{\frac{12}{5}}}\|\nabla\mathcal{U}\|^p_{L_t^{\frac{4p}{p-1}}L_x^{\frac{6p}{2p+1}}}\\
				&\lesssim t_2^{\frac{5p-4}{4}}R^{p+1}.
			\end{aligned}
		\end{equation}
		Thus, 
		\begin{align*}
			\|\nabla\Phi(\mathcal{U})\|_{L_t^{\infty}L_x^2\cap L_t^2L_x^6}\leq R+C t_2^{\frac{5p-4}{4}}R^{p+1}.
		\end{align*}
		Similarly,
		\begin{align*}
			\|\nabla(\Phi(\mathcal{U})-\Phi(\mathcal{V}))\|_{L_t^{\infty}L_x^2 \cap L_t^2L_x^6} \lesssim \| t^{\frac{3p}{2}-2} \nabla(|\mathcal{U}|^p\mathcal{U}-|\mathcal{V}|^p\mathcal{V}) \|_{L_t^{\frac{8}{7}}L_x^{\frac{12}{7}}}.
		\end{align*}
		Note that the gradient of $|u|^pu$ can be written as $F_1(u)\nabla u+F_2(u)\overline{\nabla u}$, where
		\begin{align*}
			F_1(u)=\frac{p+2}{2}|u|^p, \quad F_2(u)=\frac{p}{2}|u|^{p-2}u^2.
		\end{align*}
		We only consider the term $F_1(u)\nabla u$, since $F_2(u)\overline{\nabla u}$ is treated in the same way. Note that
		\begin{align}\label{H1:u-v}
			F_1(u)\nabla u-F_1(v)\nabla v=\frac{p+2}{2} |u|^p\nabla(u-v) + \frac{p+2}{2} (|u|^p-|v|^p)\nabla v.
		\end{align}
		By modifying the above nonlinear estimates,
		\begin{align*}
			&\|t^{\frac{3p}{2}-2}(F_1(\mathcal{U})\nabla \mathcal{U}- F_1(\mathcal{V})\nabla \mathcal{V})\|_{L_t^{\frac{8}{7}}L_x^{\frac{12}{7}}}\\
			&\lesssim (t_1-t_2)^{\frac{3p}{2}-2}  \big\| \|\nabla (\mathcal{U} - \mathcal{V})\|_{L_x^{\frac{12}{5}}} \|\mathcal{U}\|^p_{L_x^{6p}}\big\|_{L_t^{\frac{8}{7}}}\\
			&\qquad + (t_1-t_2)^{\frac{3p}{2}-2}\big\|\big(\|\mathcal U\|_{L_x^{6p}}^{p-1}+\|\mathcal V\|_{L_x^{6p}}^{p-1}\big)\|\mathcal U-\mathcal V\|_{L_x^{6p}}\|\nabla\mathcal V\|_{L_x^{12/5}} \big\|_{L_t^{\frac{8}{7}}}\\
			&\lesssim t_2^{\frac{5p-4}{4}}R^p \|\nabla(\mathcal{U}-\mathcal{V})\|_{L_t^{\infty}L_x^2 \cap L_t^2L_x^6}.
		\end{align*}
		Therefore,
		\begin{align*}
			d(\Phi(\mathcal{U}), \Phi(\mathcal{V})) \leq C t_2^{\frac{5p-4}{4}} R^{p} d(\mathcal{U}, \mathcal{V}).
		\end{align*}
		Now, by choosing $t_2$ sufficiently small, depending on $T$ and $M$, we obtain that $\Phi$ is a contraction. This finishes the proof.
	\end{proof}
	
	In view of Proposition \ref{GWP-I:prop H1}, we are able to assume that $(T^*,+\infty)$ denotes the maximal existence lifespan for \eqref{CNLS-I} such that $ \mathcal{U}\in C((T^*,\infty);\dot{H}_x^{1}(\mathbb{R}^3))$.

	\subsection{Global well-posedness}
	The global well-posedness relies on the following a priori estimate.  Under the assumptions of Theorem \ref{main-theorem1}, we obtain that for any $t\in (T^*,T_0)$,
	\begin{align}\label{derivate of pce}
		\dfrac{\mathrm{d}}{\mathrm{d}t}\mathcal{E}(t)=\frac{1}{4}(2-\frac{3p}{2})t^{1-\frac{3p}{2}}\int_{\mathbb{R}^3}|\nabla\mathcal{U}|^2\mathrm{d}x\geq 0,
	\end{align}
	where the pseudo conformal energy $\mathcal{E}(t)$ is defined in \eqref{Pseudo-conformal-energy}.
	
	By the Sobolev embedding $\dot{H}_x^{\frac{3p}{2p+4}}(\mathbb{R}^3)\hra L_x^{p+2}(\mathbb{R}^3)$, $-s_c\leq \frac{3p}{2p+4} < 1$, interpolation, and Corollary \ref{cor:local-weighted-1}, we obtain that
	\begin{align}\label{L p+2 estimate}
		\|\mathcal{U}(T_0)\|_{L_x^{p+2}(\mathbb{R}^3)} \lsm \|\mathcal{U}(T_0)\|_{\dot{H}^{-s_c}_x\cap\dot{H}_x^{1}(\mathbb{R}^3)} \lsm 1.
	\end{align}
	By \eqref{Pseudo-conformal-energy}, \eqref{derivate of pce}, \eqref{L p+2 estimate}, and Corollary \ref{cor:local-weighted-1}, we have  that for any $t\in (T^*,T_0)$,
	\EQ{
		0\le \E(t) \le \E(T_0) \lsm 1,
	}
	which implies that for any  $t\in (T^*,T_0)$,
	\begin{align}\label{a prior estimate for H1}
		\|\mathcal{U}(t)\|_{\dot H_x^1(\mathbb{R}^3)}\lsm t^{\frac{3p}{4}-1}.
	\end{align}

	%	For any $T^*<T<T_0$, by \eqref{a prior estimate for H1}, we have
	%	\begin{align*}
		%		\|\nabla\mathcal{U}(t)\|_{L_x^2}\lesssim T^{\frac{3p-4}{4}} ,\quad T<t<T_0.
		%	\end{align*}
	
	\begin{prop}[Global well-posedness in $\dot H_x^1(\mathbb{R}^3)$]\label{prop:gwp-weighted-1}
		Let the assumptions in Theorem \ref{main-theorem1} hold. Then, there exists a unique solution $\U$ of \eqref{CNLS-I} such that $\mathcal{U}\in C((0,\infty);\dot{H}^1_x(\mathbb{R}^3))$. Moreover, we have that
		\begin{align}\label{energy estimate E(t)}
			\sup_{0< t\le T_0} \E(t) \lsm 1. 
		\end{align}
	\end{prop}	
	\begin{proof}
		We prove this proposition by contradiction. Assume that $T^*>0$. Then, by \eqref{a prior estimate for H1}, for any $t\in(T^*,T_0)$,
		\EQ{
			\norm{\U(t)}_{\dot H_x^1(\R^3)} \le C_1(T^*)^{\frac{3p}{4}-1},
		}
		for some $C_1>0$. For any $\ep>0$, Proposition \ref{GWP-I:prop H1} with $T=T^*$, $t_1=T^*+\ep$, and $M=C_1(T^*)^{\frac{3p}{4}-1}$ yields that there exists $t_2$ depending on $T^*$ and $M$ such that $\mathcal{U}\in C([T^*+ \varepsilon - t_2, T^* + \varepsilon]; \dot{H}_x^1(\R^3))$. Notice that $t_2$ is actually independent of $\ep$. Taking $\varepsilon=\frac{t_2}{2}$, we obtain $T^*+\varepsilon - t_2 < T^*$, which contradicts the definition of $T^*$. Thus, $T^*=0$.
	\end{proof}

	\subsection{Additional regularity}
	In this section, we establish additional regularity of the solution $\mathcal{U}$ in $\dot{H}_x^{-s_c}(\mathbb{R}^3)$.
	\begin{comment}
		We also prove that $|\nabla|^{-s_c}\mathcal{U} \in L_t^{q} L_x^{r} ((0,\infty)\times \mathbb{R}^3)$ for any admissible pair $(q,r)$ with $q\neq\infty$.
	\end{comment}
	\begin{prop}\label{P2}
		Let the assumptions in Theorem \ref{main-theorem1} hold. Let also $\mathcal{U}$ be the global solution given in the former step such that $\mathcal{U}\in C((0,\infty);\dot{H}_x^1(\mathbb{R}^3))$. Then we have that $\mathcal{U}\in C((0,\infty);\dot{H}^{-s_c}_x\cap\dot{H}^{1}_x(\mathbb{R}^3))$.
	\end{prop}
	\begin{proof}
		It suffices to prove that for any $0<T<T_0$,
		\begin{align}\label{GWP-II:Hs_c}
			\||\nabla|^{-s_c}\mathcal{U}\|_{L_t^{\infty}L_x^{2}\cap L_t^2L_x^6([T,T_0]\times \mathbb{R}^3)}\lesssim_{T}1.
		\end{align}
		Note that $T$ can be viewed as a constant, so we omit its dependence.
		
		Denote $I:=[b-a,b]$, where $T<b\leq T_0$. Set $0<a<\frac{T}{2}$ temporarily. Assume $\mathcal{U}(b)\in\dot{H}_x^{-s_c}(\mathbb{R}^3)$. In the following, we restrict the spacetime variable $(t,x)\in I\times\R^3$. Recall the integral equation
		\begin{align*}
			\mathcal{U}(t)=S(t-b)\mathcal{U}(b)-i\int_{b}^{t}S(t-\tau)(\tau^{\frac{3p}{2}-2}|\mathcal{U}|^p\mathcal{U})\mathrm{d}\tau.
		\end{align*}
		By \eqref{a prior estimate for H1} and the same argument as in \eqref{proof-Proposition-away from 0-1}, 
		\begin{align*}
			\|\nabla\mathcal{U}\|_{L_t^{\infty}L_x^2 \cap L_t^2L_x^6}& \leq  C\|\nabla\mathcal{U}(b)\|_{L_x^2(\mathbb{R}^3)} + C a^{\frac{5p-4}{4}}\|\nabla\mathcal{U}\|_{L_t^{\infty}L_x^2 \cap L_t^2L_x^6}^{p+1}\\
			&\leq C T^{\frac{3p-4}{4}} + C a^{\frac{5p-4}{4}} \|\nabla\mathcal{U}\|_{L_t^{\infty}L_x^2 \cap L_t^2L_x^6}^{p+1}.
		\end{align*}
		Temporarily set $a=a(T)$ sufficiently small so that
		\begin{align}\label{thm1-gwp2-1order}
			\|\nabla\mathcal{U}\|_{L_t^{\infty}L_x^2 \cap L_t^2L_x^6}\lesssim 1.
		\end{align}
		For the inhomogeneous term, applying again the same argument as in \eqref{proof-Proposition-away from 0-1} and using \eqref{thm1-gwp2-1order}, we obtain
		\begin{align*}
			&\left\||\nabla|^{-s_c}\int_{b}^{t}S(t-\tau)(\tau^{\frac{3p}{2}-2}|\mathcal{U}|^p\mathcal{U})\mathrm{d}\tau\right\|_{L_t^{\infty}L_x^2\cap L_t^2L_x^6}\\
			&\lesssim|a|^{\frac{5p-4}{4}}\|\nabla\mathcal{U}\|^p_{L_t^{\frac{4p}{p-1}}L_x^{\frac{6p}{2p+1}}}\||\nabla|^{-s_c}\mathcal{U}\|_{L_t^{8}L_x^{\frac{12}{5}}}\\
			&\lesssim |a|^{\frac{5p-4}{4}}\||\nabla|^{-s_c}\mathcal{U}\|_{L_t^{\infty}L_x^2 \cap L_t^2L_x^6}.
		\end{align*}
		By the Strichartz estimate, 
		\begin{align*}
			\||\nabla|^{-s_c}\mathcal{U}\|_{L_t^{\infty}L_x^2 \cap L_t^2L_x^6} \leq  C\||\nabla|^{-s_c}\mathcal{U}(b)\|_{L_x^2(\mathbb{R}^3)} + C|a|^{\frac{5p-4}{4}} \||\nabla|^{-s_c}\mathcal{U}\|_{L_t^{\infty}L_x^2 \cap L_t^2L_x^6}. 
		\end{align*}
		Thus, there exists $a=a(T)<1$ such that $C|a|^{\frac{5p-4}{4}}< \frac{1}{2}$ and
		\begin{align*}
			\||\nabla|^{-s_c}\mathcal{U}\|_{L_t^{\infty}L_x^2\cap L_t^2L_x^6([b-a, b]\times\R^3)}\lesssim \||\nabla|^{-s_c}\mathcal{U}(b)\|_{L_x^2(\mathbb{R}^3)}.
		\end{align*} 
		
		Split $[T,T_0]$ into small consecutive intervals $I_j$, with $j=1,2,\cdots,m=m(T)$, such that $|I_j|\leq a$. By induction on $I_j$, we are able to prove \eqref{GWP-II:Hs_c}.
	\end{proof}
	Therefore, we obtain that \eqref{CNLS-I} admits a global solution $\mathcal{U}\in C((0,\infty); \dot{H}^{-s_c}_x\cap\dot{H}^{1}_x(\mathbb{R}^3))$. The preceding estimate is not uniform as $T\downarrow0$. We next use the pseudo conformal energy to obtain a bound on the entire interval $(0,T_0)$. More precisely, we are able to prove that 
	\begin{align}\label{esti-global bound for (0,T_0)}
		\||\nabla|^{-s_c}\mathcal{U}\|_{L_t^{\infty}L_x^{2}\cap L_t^2L_x^6((0,T_0)\times \mathbb{R}^3)} \lesssim 1.
	\end{align}
	
	We first prove that $|\nabla|^{-s_c}\U$ has some spacetime estimate through pseudo conformal energy: $|\nabla|^{-s_c}\mathcal{U} \in L_t^{\frac{4p}{5p-4}} L_x^{\frac{3p}{2-p}} ((0,\infty)\times \mathbb{R}^3)$, where $(\frac{4p}{5p-4}, \frac{3p}{2-p})$ is admissible.
	
	\begin{lem}\label{lem--s_c-global-bound}
		Let the assumptions in Theorem \ref{main-theorem1} hold. Let also $\mathcal{U}$ be the global solution given in the former step such that $\mathcal{U}\in C((0,\infty);\dot{H}_x^{-s_c} \cap  \dot{H}_x^1(\mathbb{R}^3))$. Then 
		\begin{align*}
			\||\nabla|^{-s_c}\mathcal{U}\|_{L_t^{\frac{4p}{5p-4}} L_x^{\frac{3p}{2-p}} ((0,\infty)\times \mathbb{R}^3)} \lesssim 1.
		\end{align*}
	\end{lem}
	\begin{proof}
		Since $(\frac{4p}{5p-4}, \frac{3p}{2-p})$ is an admissible pair, it follows from Corollary \ref{cor:local-weighted-1} that 
		\begin{align}\label{estimate for (T_0, infty)}
			\||\nabla|^{-s_c}\mathcal{U}\|_{L_t^{\frac{4p}{5p-4}} L_x^{\frac{3p}{2-p}} ([T_0, \infty) \times \mathbb{R}^3)} \lesssim \||\nabla|^{-s_c}\mathcal{U}\|_{L_t^{\frac{4p}{5p-4},2} L_x^{\frac{3p}{2-p}} ([T_0, \infty) \times \mathbb{R}^3)}  \lesssim 1.
		\end{align}
		For all $2\leq q<\infty$, by \eqref{derivate of pce} and \eqref{energy estimate E(t)},
		\begin{equation}\label{estimate for E(T_0)}
			\begin{aligned}
				&\int_{0}^{T_0}\left(t^{\frac{1}{2}(2-\frac{3p}{2})}\|\nabla\mathcal{U}\|_{L_x^2(\R^3)}\right)^q\frac{\mathrm{d}t}{t}\\
				&=\int_{0}^{T_0}\left(t^{2-\frac{3p}{2}}\|\nabla\mathcal{U}\|^2_{L_x^2(\R^3)}\right)^{\frac{q-2}{2}}\left(t^{1-\frac{3p}{2}}\|\nabla\mathcal{U}\|^2_{L_x^2(\R^3)}\right)\mathrm{d}t\\
				&\lesssim \mathcal{E}(T_0)^{\frac{q-2}{2}} \int_{0}^{T_0}t^{1-\frac{3p}{2}}\|\nabla\mathcal{U}\|^2_{L_x^2(\R^3)}\mathrm{d}t \lesssim \mathcal{E}(T_0)^{\frac{q-2}{2}} (\mathcal{E}(T_0) - \mathcal{E}(0))\lesssim 1.
			\end{aligned}
		\end{equation}
		Thus, by H\"older's inequality, Sobolev's inequality, and \eqref{estimate for E(T_0)} for $q=4$, 
		\begin{equation}\label{estimate for (0, T_0)}
			\begin{aligned}
				&\||\nabla|^{-s_c}\mathcal{U}\|_{L_t^{\frac{4p}{5p-4}} L_x^{\frac{3p}{2-p}} ((0,T_0) \times \mathbb{R}^3)}\\
				&\lesssim \big\|t^{-\frac{1}{4}+\frac{1}{2}(2-\frac{3p}{2})}|\nabla|^{-s_c}\mathcal{U}\big\|_{L_t^{4}L_x^{\frac{3p}{2-p}}((0,T_0)\times \mathbb{R}^3)}\big\|t^{\frac{1}{4}-\frac{1}{2}(2-\frac{3p}{2})}\big\|_{L_t^{\frac{p}{p-1}}(0,T_0)}\\
				&\lesssim \big\|t^{-\frac{1}{4}+\frac{1}{2}(2-\frac{3p}{2})}\|\nabla\mathcal{U}\|_{L_x^2(\mathbb{R}^3)}\big\|_{L_t^{4}(0,T_0)}\lesssim 1.
			\end{aligned}
		\end{equation}
		This lemma follows from \eqref{estimate for (T_0, infty)} and \eqref{estimate for (0, T_0)}.
	\end{proof}
	
	\begin{lem}\label{Lemma-short interval bound}
		Let $\mathcal{U}\in C((0,\infty); \dot{H}^{-s_c}_x\cap\dot{H}^{1}_x(\mathbb{R}^3))$ be a solution of \eqref{CNLS-I}. Assume that $\varepsilon>0$ is sufficiently small and that
		\begin{align*}
			\||\nabla|^{-s_c} \mathcal{U}\|_{L_t^{\frac{4p}{5p-4}} L_x^{\frac{3p}{2-p}} (I \times \mathbb{R}^3)}<\varepsilon,
		\end{align*}
		where $I=(0, \tau_0)$ or $(\tau_{1}, \tau_0)$. Then, for any admissible pair $(q,r)$ with $q\neq \infty$, 
		\begin{align}\label{esti-lem-(0,tau_0)}
			\||\nabla|^{-s_c}\mathcal{U}\|_{L_t^{\infty}L_x^2\cap L_t^{q,2}L_x^r(I\times\R^3)} + \|\nabla\mathcal{U}\|_{L_t^{\infty}L_x^2\cap L_t^{q,2}L_x^r(I\times\R^3)}\lesssim_{\tau_0} 1.
		\end{align}
	\end{lem}
	\begin{proof}
		We consider the following integral equation
		\begin{align*}
			\mathcal{U}(t)=S(t-\tau_0)\mathcal{U}(\tau_0)-i\int_{\tau_0}^{t}S(t-\tau)(\tau^{\frac{3p}{2}-2}|\mathcal{U}|^p\mathcal{U})\mathrm{d}\tau.
		\end{align*}
		In the following, we restrict the spacetime variable $(t,x) \in I\times\mathbb{R}^3$. By the Strichartz estimate, H\"older's and Sobolev's inequalities,
		\begin{equation}\label{estimate-Lemma-J(t)-s_c-1}
			\begin{aligned}
				&\||\nabla|^{-s_c}\mathcal{U}\|_{L_t^{\frac{8}{p},2}L_x^{\frac{12}{6-p}}}\\
				&\lesssim\|\mathcal{U}(\tau_0)\|_{\dot{H}_x^{-s_c}(\R^3)}+\||\nabla|^{-s_c} (t^{\frac{3p}{2}-2}|\mathcal{U}|^p\mathcal{U})\|_{L_t^{\frac{8}{8-p},2}L_x^{\frac{12}{p+6}}}\\
				&\lesssim\|\mathcal{U}(\tau_0)\|_{\dot{H}_x^{-s_c}(\R^3)}+\|t^{s_cp}\|_{L_t^{-\frac{1}{s_cp},\infty}}\||\nabla|^{-s_c}\mathcal{U}\|^p_{L_t^{\frac{4p}{5p-4}} L_x^{\frac{3p}{2-p}}}\||\nabla|^{-s_c}\mathcal{U}\|_{L_t^{\frac{8}{p},2}L_x^{\frac{12}{6-p}}}\\
				&\lesssim \|\mathcal{U}(\tau_0)\|_{\dot{H}_x^{-s_c}(\R^3)}+\varepsilon^{p} \||\nabla|^{-s_c}\mathcal{U}\|_{L_t^{\frac{8}{p},2}L_x^{\frac{12}{6-p}}}.
			\end{aligned}
		\end{equation}
		Note that $0<\varepsilon\ll 1$ is sufficiently small. Thus, 
		\begin{align}\label{est-lemma-11111}
			\||\nabla|^{-s_c}\mathcal{U}\|_{L_t^{\frac{8}{p},2}L_x^{\frac{12}{6-p}}} \lesssim \|\mathcal{U}(\tau_0)\|_{\dot{H}_x^{-s_c}(\mathbb{R}^3)}.
		\end{align}
		Applying the same argument as in \eqref{estimate-Lemma-J(t)-s_c-1} and using \eqref{est-lemma-11111}, we obtain
		\begin{align}\label{esti-lem-(0,tau_0)--s_c}
			\||\nabla|^{-s_c}\mathcal{U}\|_{L_t^{\infty}L_x^2\cap L_t^{q,2}L_x^r} \lesssim \|\mathcal{U}(\tau_0)\|_{\dot{H}_x^{-s_c}(\mathbb{R}^3)}.
		\end{align}
		Similarly,
		\begin{align}\label{esti-lem-(0,tau_0)-1}
			\|\nabla\mathcal{U}\|_{L_t^{\infty}L_x^2\cap L_t^{q,2}L_x^r} \lesssim \|\mathcal{U}(\tau_0)\|_{\dot{H}_x^1(\mathbb{R}^3)}.
		\end{align}
		This lemma follows from \eqref{esti-lem-(0,tau_0)--s_c} and \eqref{esti-lem-(0,tau_0)-1}.
	\end{proof}
	
	We are now in a position to prove \eqref{esti-global bound for (0,T_0)}. Indeed, estimate \eqref{esti-global bound for (0,T_0)} follows immediately from the proposition stated below:
	\begin{prop}\label{prop--s_c-1-gloabl bound}
		Let the assumptions in Theorem \ref{main-theorem1} hold. Let $\mathcal U$ be the global solution constructed above. Then, for any admissible pair $(q,r)$ with $q\neq\infty$, 
		\begin{align}\label{estimate-radial-global-1 and s_c}
			\||\nabla|^{-s_c}\mathcal{U}\|_{L_t^{\infty}L_x^2\cap L_t^{q,2}L_x^r((0,\infty)\times\R^3)} + \|\nabla\mathcal{U}\|_{L_t^{\infty}L_x^2\cap L_t^{q,2}L_x^r((0,\infty)\times\R^3)} \lesssim 1.
		\end{align}
	\end{prop}
	\begin{proof}
		For any $0<\varepsilon\ll 1$, it follows from Lemma \ref{lem--s_c-global-bound} that we can split $(0,T_0)$ into $0=\tau_0<\tau_1<\cdots<\tau_{m(\varepsilon)-1}<\tau_{m(\varepsilon)}=T_0$ such that 
		\begin{align*}
			\||\nabla|^{-s_c} \mathcal{U} \|_{L_t^{\frac{4p}{5p-4}} L_x^{\frac{3p}{2-p}} (I_j \times \mathbb{R}^3)}<\varepsilon,
		\end{align*}
		where $I_j=(\tau_j, \tau_{j+1}]$, $0\leq j < m(\varepsilon)$. By Lemma \ref{Lemma-short interval bound},
		\begin{align*}
			\||\nabla|^{-s_c}\mathcal{U}\|_{L_t^{\infty}L_x^2\cap L_t^{q,2}L_x^r(I_j\times\R^3)} + \|\nabla\mathcal{U}\|_{L_t^{\infty}L_x^2\cap L_t^{q,2}L_x^r(I_j\times\R^3)}\lesssim_{\tau_{j+1}} 1.
		\end{align*}
		Summing over these intervals $I_j$, we obtain
		\begin{align}\label{esti-s_c-1-(0,T_0)}
			\||\nabla|^{-s_c}\mathcal{U}\|_{L_t^{\infty}L_x^2\cap L_t^{q,2}L_x^r((0,T_0)\times\R^3)} + \|\nabla\mathcal{U}\|_{L_t^{\infty}L_x^2\cap L_t^{q,2}L_x^r((0,T_0)\times\R^3)}\lesssim_{\tau_{1},\cdots,\tau_{m(\varepsilon)}} 1.
		\end{align}
		Note that $\tau_{j+1}$, $0\leq j < m(\varepsilon)$, can be viewed as constants. Thus, by \eqref{esti-bound for (T_0, infty)} and \eqref{esti-s_c-1-(0,T_0)}, we obtain \eqref{estimate-radial-global-1 and s_c}.
	\end{proof}
	
	Therefore, applying the inverse pseudo conformal transform $u=\mathcal{T}^{-1}\mathcal{U}$, we obtain that \eqref{NLS} admits a unique solution $S(-t)u\in C((0,\infty);\mathcal{F}\dot{H}_x^{-s_c}\cap\mathcal{F}\dot{H}_x^1(\R^3))$ and for any admissible pair $(q,r)$ with $q\neq\infty$,
	\begin{align}\label{esti-u-global-bound}
		\||J(t)|^{-s_c}u\|_{L_t^{\infty}L_x^2\cap L_t^{q,2}L_x^r((0,\infty)\times\R^3)} + \|J(t)u\|_{L_t^{\infty}L_x^2\cap L_t^{q,2}L_x^r((0,\infty)\times\R^3)}\lesssim 1.
	\end{align}
	The Lorentz-space bounds follow by applying the Lorentz-space Strichartz estimates to the Duhamel formula for $u$ and partitioning the time interval according to the spacetime bounds obtained above.

	\subsection{Scattering} 
	Next, we prove the scattering. Since our case includes the 3D Strauss exponent $p=1$, we basically follow the argument by Nakanishi and Ozawa \cite{2002-Nakanishi-Ozawa}.
	\begin{prop}[Scattering]\label{scattering}
		Let $u$ be the unique global solution to \eqref{NLS} given in Theorem \ref{main-theorem1}. Then the solution scatters in $\mathcal{F}\dot{H}_x^{-s_c}\cap\mathcal{F}\dot{H}_x^{1}(\mathbb{R}^3)$, that is, there exist $u_{\pm}\in \mathcal{F}\dot{H}_x^{-s_c}\cap\mathcal{F}\dot{H}_x^{1}(\mathbb{R}^3)$ such that 
		\begin{align}\label{scattering expression1}
			\lim_{t\rightarrow \pm\infty}\|S(-t)u-u_{\pm}\|_{\mathcal{F}\dot{H}_x^{-s_c}\cap\mathcal{F}\dot{H}_x^{1}(\mathbb{R}^3)}=0.
		\end{align}  
	\end{prop}
	
	\begin{proof}
		We consider only the case $t\rightarrow +\infty$, since the case $t\rightarrow -\infty$ is treated in the same way. It follows from \eqref{esti-u-global-bound} that for any $0<\varepsilon\ll 1$, there exists $\tau_0>1$ such that
		\begin{align}\label{scattering: estimate for H1-s_c}
			\||J(t)|^{-s_c} u\|_{L_t^{\frac{4p}{5p-4}} L_x^{\frac{3p}{2-p}} ((\tau_0,\infty) \times \mathbb{R}^3)} + \||J(t)|^{-s_c}u\|_{L_t^{\frac{8}{p},2}L_x^{\frac{12}{6-p}}((\tau_0,\infty)\times \mathbb{R}^3)} <\varepsilon,
		\end{align}
		and
		\begin{align}\label{scattering: estimate for H1-s_c-2}
			\|J(t)u\|_{L_t^{\frac{8}{p},2}L_x^{\frac{12}{6-p}}((\tau_0,\infty)\times \mathbb{R}^3)} <\varepsilon.
		\end{align}
		Note that for any $1<\tau_0<\tau_1<\tau_2$,
		\begin{align}\label{scattering-S(-t)u(t)}
			S(-\tau_2)u(\tau_2)-S(-\tau_1)u(\tau_1) = -i \int_{\tau_1}^{\tau_2}S(-\tau)(|u|^pu)\mathrm{d}\tau.
		\end{align}
		Then, by the same argument as in \eqref{estimate-J(t)-inhomogeneous},
		\begin{align*}
			&\|S(-\tau_2)u(\tau_2)-S(-\tau_1)u(\tau_1)\|_{\mathcal{F}\dot{H}^{-s_c}_x(\mathbb{R}^3)}\\
			&= \left\|\int_{\tau_1}^{\tau_2}S(-\tau)|J(\tau)|^{-s_c}(|u|^pu)\mathrm{d}\tau\right\|_{L_x^2(\mathbb{R}^3)} \\
			&\lesssim\||J(t)|^{-s_c}(|u|^pu)\|_{L_t^{\frac{8}{8-p},2}L_x^{\frac{12}{p+6}}((\tau_1,\tau_2)\times\mathbb{R}^3)}\\
			&\lesssim  \||J(t)|^{-s_c}u\|^p_{L_t^{\frac{4p}{5p-4}} L_x^{\frac{3p}{2-p}}((\tau_1,\tau_2)\times\mathbb{R}^3)}\||J(t)|^{-s_c}u\|_{L_t^{\frac{8}{p},2}L_x^{\frac{12}{6-p}}((\tau_1,\tau_2)\times\mathbb{R}^3)}.
		\end{align*}
		Similarly, 
		\begin{align*}
			&\|S(-\tau_2)u(\tau_2)-S(-\tau_1)u(\tau_1)\|_{\mathcal{F}\dot{H}^{1}_x(\mathbb{R}^3)}\\
			&\lesssim \||J(t)|^{-s_c}u\|^p_{L_t^{\frac{4p}{5p-4}} L_x^{\frac{3p}{2-p}}((\tau_1,\tau_2)\times\mathbb{R}^3)}\|J(t)u\|_{L_t^{\frac{8}{p},2}L_x^{\frac{12}{6-p}}((\tau_1,\tau_2)\times\mathbb{R}^3)}.
		\end{align*}
		Therefore, by \eqref{scattering: estimate for H1-s_c} and \eqref{scattering: estimate for H1-s_c-2},
		\begin{align*}
			\|S(-\tau_2)u(\tau_2)-S(-\tau_1)u(\tau_1)\|_{\mathcal{F}\dot{H}_x^{-s_c}\cap\mathcal{F}\dot{H}^{1}_x(\mathbb{R}^3)} \lesssim \varepsilon.
		\end{align*}
		This finishes the proof of \eqref{scattering expression1}.
	\end{proof}
	
	\vspace{2cm}
	
	\section{Proof of Theorem \ref{main-theorem2}}
	
	\vspace{0.5cm}

	Similar to Theorem \ref{main-theorem1}, we only prove Theorem \ref{main-theorem2} for $t>0$.

	\subsection{Local theory I: near the origin} For radial data, we first study the local well-posedness in $ \dot{H}_x^{s_c}\cap\mathcal{F}\dot{H}_x^1(\mathbb{R}^3)$ at the origin. Recall the integral equation
	\begin{align}\label{radial: integral equation}
		u=S(t)u_0-i\int_{0}^{t}S(t-\tau)(|u|^pu)\mathrm{d}\tau.
	\end{align}
	\begin{prop}[Local well-posedness at $t=0$]\label{radial:LWP-I}
		Let the assumptions in Theorem \ref{main-theorem2} hold. Then, there exists $I=[0,t_0]$ such that the equation \eqref{NLS} admits a unique solution $S(-t)u\in C(I;\dot{H}_x^{s_c}\cap\mathcal{F}\dot{H}_x^1(\mathbb{R}^3))$. Moreover, for any radial-admissible pair $(q,r)$ with $\frac{2}{q}+\frac{3}{r}=\frac{3}{2}-s_c$,
		\begin{align}\label{prop-radial-lwp-I}
			\|u\|_{L_t^{\infty}\dot{H}_x^{s_c}\cap L_t^{q}L_x^{r}(I\times \mathbb{R}^3)} + \|J(t)u\|_{L_t^{\infty}L_x^2\cap L_t^{2}L_x^6(I\times\R^3)}\lesssim \|u_0\|_{\dot{H}_x^{s_c}\cap\mathcal{F}\dot{H}_x^1(\mathbb{R}^3)}.
		\end{align}
	\end{prop}
	\begin{proof}
		We denote
		\begin{align}\label{radial-LWP-I-(q1,r1)}
			q_1= \frac{2p(p+2)}{4-p};\ \wt{q}_1=\frac{2p(p+2)}{3p^2+p-4}; \ q_2 = \frac{4(p+2)}{3p}; \ r_1=\wt{r}_1=r_2=p+2.  
		\end{align}
		A direct calculation shows that $(q_1,r_1)$ and $(\wt{q}_1, \wt{r}_1)$ are radial-admissible, and $(q_2,r_2)$ is admissible. In addition,
		\begin{align*}
			\frac{2}{q_1}+\frac{3}{r_1} = \frac{3}{2} - s_c;\quad \frac{2}{\wt{q}_1}+\frac{3}{\wt{r}_1} = \frac{3}{2}+s_c.
		\end{align*}
		In the following, we restrict the spacetime variable $(t,x) \in I\times\mathbb{R}^3$. By Lemma \ref{RSE:radial strichartz estimate}, we have
		\begin{align*}
			&\|u\|_{L_t^{\infty}\dot{H}_x^{s_c}\cap L_t^{q}L_x^{r}}\leq C \|u_0\|_{\dot{H}_x^{s_c}(\R^3)}+ C\||u|^pu\|_{L_t^{\wt{q}_1^{\prime}}L_x^{\wt{r}_1^{\prime}}}\leq C\|u_0\|_{\dot{H}_x^{s_c}(\R^3)}+C\|u\|^{p+1}_{L_t^{q_1}L_x^{r_1}},\\
			&\|u\|_{ L_t^{q_1}L_x^{r_1}}\leq C \|S(t)u_0\|_{L_t^{q_1}L_x^{r_1}}+C\||u|^pu\|_{L_t^{\wt{q}_1^{\prime}}L_x^{\wt{r}_1^{\prime}}}\leq C \|S(t)u_0\|_{L_t^{q_1}L_x^{r_1}}+ C\|u\|^{p+1}_{L_t^{q_1}L_x^{r_1}},
		\end{align*}
		and
		\begin{align}\label{radial-LWP-I-estimate-1}
			\|u-v\|_{L_t^{\infty}\dot{H}_x^{s_c}\cap L_t^{q_1}L_x^{r_1}}&\leq C  (\|u\|^p_{L_t^{q_1}L_x^{r_1}}+\|v\|^p_{L_t^{q_1}L_x^{r_1}})\|u-v\|_{L_t^{q_1}L_x^{r_1}}.
		\end{align}
		We consider the resolution space
		\begin{align*}
			E(R,\delta):= \{ u\in C(I;\dot{H}_x^{s_c}(\R^3)):\|u\|_{L_t^{q_1}L_x^{r_1}}\leq 2\delta, \|u\|_{L_t^{\infty}\dot{H}_x^{s_c}}\leq 2R \}.
		\end{align*}
		Let $R:=C\|u_0\|_{\dot{H}_x^{s_c}(\R^3)}$ and choose $\delta>0$ sufficiently small. It follows from the contraction mapping argument that the equation \eqref{NLS} admits a unique solution $u\in C(I;\dot{H}_x^{s_c}(\mathbb{R}^3))$.
		
		Next, we prove that $J(t)u \in L_t^{\infty}L_x^2\cap L_t^{2}L_x^6(I\times\R^3)$, and in particular, $S(-t)u\in C(I;\dot{H}_x^{s_c}\cap\mathcal{F}\dot{H}_x^1(\mathbb{R}^3))$. By \eqref{radial: integral equation} and the Strichartz estimate,
		\begin{equation}\label{radial-LWP-J(t)}
			\begin{aligned}
				\|J(t)u\|_{L_t^{\infty}L_x^2\cap L_t^{2}L_x^6}&\lesssim \|u_0\|_{\mathcal{F}\dot{H}_x^1(\R^3)}+ \|J(t)(|u|^pu)\|_{L_t^{q'_2}L_x^{r'_2}}\\
				&\lesssim \|u_0\|_{\mathcal{F}\dot{H}_x^1(\R^3)}+\|u\|^p_{L_t^{q_1}L_x^{r_1}}\|J(t)u\|_{L_t^{q_2}L_x^{r_2}}\\
				&\lesssim \|u_0\|_{\mathcal{F}\dot{H}_x^1(\R^3)}+\delta^p \|J(t)u\|_{L_t^{\infty}L_x^2\cap L_t^{2}L_x^6}.
			\end{aligned}
		\end{equation}
		Note that $\delta$ is sufficiently small. Thus, $\|J(t)u\|_{L_t^{\infty}L_x^2\cap L_t^{2}L_x^6(I\times\R^3)}\lesssim \|u_0\|_{\mathcal{F}\dot{H}_x^1(\R^3)}$.
	\end{proof}
	\begin{remark}
		In the following, we will regard $t_0$, $\|u_0\|_{\dot{H}^{s_c}_x(\mathbb{R}^3)}$, and $\norm{u_0}_{\mathcal{F}\dot{H}_x^{1}(\R^3)}$ as constants, and omit dependence on these quantities for simplicity.
	\end{remark}
	
	We apply the pseudo conformal transform $\mathcal{T}$ to \eqref{NLS}, and then $\mathcal{U}=\TT u$ satisfies \eqref{CNLS-I}
	with asymptotic condition
	\begin{align*}
		\lim_{t\rightarrow +\infty}\|S(-t)\mathcal{U}(t)- \mathcal{U}_+ \|_{\mathcal{F}\dot{H}_x^{s_c}\cap \dot{H}_x^1(\mathbb{R}^3)}=0,
	\end{align*}
	where $\mathcal{U}_+:=\mathcal{F}^{-1}\bar{u}_0\in \mathcal{F}\dot{H}_x^{s_c}\cap \dot{H}_x^1(\mathbb{R}^3)$.
	
	As a consequence of Proposition \ref{radial:LWP-I}, we have
	\begin{cor}\label{cor-radial:local-weighted-1}
		There exists $T_0=t_0^{-1}>0$ such that the equation \eqref{CNLS-I} admits a unique solution $ S(-t)\mathcal{U}\in C([T_0,\infty);\mathcal{F}\dot{H}_x^{s_c}\cap\dot{H}_x^1(\mathbb{R}^3)). $
	\end{cor}

	\subsection{Local theory II: away from the origin} 
	There is a slight difference from the proof of Theorem \ref{main-theorem1}. When applying the radial Strichartz estimate, it is more convenient to deal with the solution $u$ rather than $\U$. %We are not plan to extend the solution in Corollary \ref{cor-radial:local-weighted-1} from $T_0$ backward to 0. If we follow the steps of the proof of Theorem \ref{main-theorem1}, we need to prove that $S(-t)\mathcal{U}\in C((0,\infty); \mathcal{F}\dot{H}_x^{s_c}(\mathbb{R}^3))$. However, it is not obvious to obtain by the radial Strichartz estimate. We expect to use the radial Strichartz estimate to deal with the case $\dot{H}_x^{s_c}(\mathbb{R}^3)$, rather than $\mathcal{F}\dot{H}_x^{s_c}(\mathbb{R}^3)$.
	
	Therefore, we reduce the problem to extending the solution in Proposition \ref{radial:LWP-I} from $t_0$ to $\infty$. We need the following local result with initial time greater than $t_0$, only requiring the initial data in $ \mathcal{F} \dot H_x^1(\mathbb{R}^3)$. This result is also equivalent to the blow-up criterion by $\mathcal{F} \dot H_x^1(\mathbb{R}^3)$-norm.
	
	\begin{prop}\label{GWP-I:radial-prop H1}
		Let the assumptions in Theorem \ref{main-theorem2} hold, and fix any $t<\infty$. Assume that there exists $M>0$ and some $t_1\in(t_0, t)$ such that
		\begin{align*}
			\|S(-t_1)u(t_1)\|_{\mathcal{F}\dot{H}_x^1(\mathbb{R}^3)} \leq M.
		\end{align*}
		Then there exists some $t_2>0$ depending on $M$ such that the equation \eqref{NLS} admits a unique solution $ S(-t) u \in C([t_1, t_1+t_2];\mathcal{F}\dot{H}^1_x(\mathbb{R}^3))$.
	\end{prop}
	\begin{proof}
		We consider the map
		\begin{align*}
			\Psi(u):=S(t-t_1)u(t_1)-i\int_{t_1}^{t}S(t-\tau)(|u|^pu)\mathrm{d}\tau.
		\end{align*}
		Take $t_2$ that will be defined later. Set $I:=[t_1, t_1+t_2]$. In the following, we restrict the spacetime variable $(t,x)\in I\times\R^3$. By the Strichartz estimate, 
		\begin{align*}
			\|J(t)S(t-t_1)u(t_1)\|_{L_t^{\infty}L_x^2\cap L_t^2L_x^6}\leq C\|S(-t_1)u(t_1)\|_{\mathcal{F}\dot{H}_x^1(\mathbb{R}^3)} \leq CM.
		\end{align*}
		We denote the right hand side of the above estimate as  $R:=CM$. Define the resolution space
		\begin{align*}
			E(R):=\{S(-t)u(t)\in C(I;\mathcal{F}\dot{H}^1_x(\R^3)):\|J(t)u(t)\|_{L_t^{\infty}L_x^2\cap L_t^2L_x^6}\leq 2R\}.
		\end{align*}
		Applying the same argument as in \eqref{estimate-J(t)-inhomogeneous} and \eqref{proof-Proposition-away from 0-1},
		\begin{equation}\label{radial-inhomogeneous-J(t)u}
			\begin{aligned}
				\left\|J(t)\int_{t_1}^{t}S(t-\tau)(|u|^pu)\mathrm{d}\tau\right\|_{L_t^{\infty}L_x^2\cap L_t^2L_x^6} &\lesssim \|J(t)(|u|^pu)\|_{L_t^{\frac{8}{7}}L_x^{\frac{12}{7}}}\\
				&\lesssim t_1^{-p}t_2^{\frac{4-p}{4}} \|J(t)u\|^{p+1}_{L_t^{\infty}L_x^2\cap L_t^2L_x^6}\\
				&\lesssim t_0^{-p}t_2^{\frac{4-p}{4}} \|J(t)u\|^{p+1}_{L_t^{\infty}L_x^2\cap L_t^2L_x^6}.
			\end{aligned}
		\end{equation}
		Thus,
		\begin{align*}
			\|J(t)\Psi(u)\|_{L_t^{\infty}L_x^2\cap L_t^2L_x^6} \leq R + Ct_2^{\frac{4-p}{4}} R^{p+1}.
		\end{align*}
		Similarly,
		\begin{align*}
			\|J(t)(\Psi(u)-\Psi(v))\|_{L_t^{\infty}L_x^2\cap L_t^2L_x^6} \leq Ct_2^{\frac{4-p}{4}} R^{p}\|J(t)(u-v)\|_{L_t^{\infty}L_x^2\cap L_t^2L_x^6}.
		\end{align*}
		Therefore, by choosing $t_2$ sufficiently small, depending on $M$, and applying the contraction mapping argument, we finish the proof.
	\end{proof}

	\subsection{Global well-posedness}
	In view of Proposition \ref{GWP-I:radial-prop H1}, we are able to assume that $(0,t^*)$ denotes the maximal existence lifespan for \eqref{NLS} such that $ S(-t)u(t)\in C((0,t^*);\mathcal{F}\dot{H}_x^{1}(\mathbb{R}^3))$. Moreover, we assume that $(T^*,\infty)$ denotes the maximal existence lifespan for \eqref{CNLS-I} such that $ \mathcal{U}(t)\in C((T^*,\infty);\dot{H}_x^{1}(\mathbb{R}^3))$, where $T^*=\frac{1}{t^*}$. 
	
	A direct calculation shows that
	\begin{align}\label{RGPE}
		\|\mathcal{U}(t,\cdot)\|_{L_x^r(\mathbb{R}^3)}=t^{-\frac{3}{2}+\frac{3}{r}}\|u(t^{-1},\cdot)\|_{L_x^r(\mathbb{R}^3)},
	\end{align}
	and
	\begin{align}\label{RGPE-2}
		|\nabla|^s\mathcal{U} = |\nabla|^s\mathcal{T}u = \mathcal{T}|J(t)|^su, \quad s>0.
	\end{align}
	
	The global well-posedness relies on the following a priori estimate for $u$. First, we give an a priori estimate for $\mathcal{U}$, then use \eqref{RGPE} and \eqref{RGPE-2} to obtain the estimate for $u$.

	It follows from Proposition \ref{radial:LWP-I} that
	\begin{align*}
		\|u\|_{L_t^{\frac{2p(p+2)}{4-p}}L_x^{p+2}((0,t_0)\times \mathbb{R}^3)} \lesssim 1.
	\end{align*}
	Then, there exists $0<\wt{t}_0\leq t_0$ such that $\|u(\wt{t}_0,\cdot)\|_{L_x^{p+2}(\mathbb{R}^3)}\lesssim 1$. By \eqref{RGPE},
	\begin{align}\label{radial: L p+2 estimate}
		\|\mathcal{U}(\widetilde{T}_0)\|_{L_x^{p+2}(\mathbb{R}^3)}\lesssim 1,
	\end{align}
	where $\widetilde{T}_0=\wt{t}_0^{-1}\geq T_0$. 
	Under the assumptions of Theorem \ref{main-theorem2}, \eqref{derivate of pce} also holds. Thus, by \eqref{Pseudo-conformal-energy}, \eqref{derivate of pce}, \eqref{radial: L p+2 estimate}, and Corollary \ref{cor-radial:local-weighted-1}, we have that for any $t\in (T^*,\wt{T}_0)$,
	\begin{align}\label{radial-E(t)}
		0\leq \mathcal{E}(t) \leq \mathcal{E}(\wt{T}_0) \lesssim 1,
	\end{align}
	which implies that for any $t\in (T^*,T_0)$,
	\begin{align}\label{RG1}
		\|\nabla\mathcal{U}(t)\|_{L_x^2(\mathbb{R}^3)}\lesssim t^{\frac{3p}{4}-1}.
	\end{align}
	By \eqref{RGPE} and \eqref{RGPE-2},
	\begin{align}\label{RG2}
		\|S(-t)u(t)\|_{\mathcal{F}\dot{H}_x^1(\mathbb{R}^3)} = \|J(t)u(t)\|_{L_x^2(\mathbb{R}^3)} = \|\mathcal{T}J(t^{-1})u(t^{-1})\|_{L_x^2(\mathbb{R}^3)} = \|\nabla\mathcal{U}(t^{-1})\|_{L_x^2(\mathbb{R}^3)}.
	\end{align}
	It follows from \eqref{RG1} and \eqref{RG2} that for any $t\in (t_0, t^*)$,
	\begin{align}\label{radial-a priori estimate for u}
		\|S(-t)u(t)\|_{\mathcal{F}\dot{H}_x^1(\mathbb{R}^3)} \lesssim t^{1-\frac{3p}{4}}.
	\end{align}

	\begin{prop}[Global well-posedness in $\F\dot{H}_x^1(\R^3)$]
		Let the assumptions in Theorem \ref{main-theorem2} hold. Then, there exists a unique solution $u$ of \eqref{NLS} such that $$S(-t)u\in C((0,\infty); \mathcal{F}\dot{H}_x^1(\mathbb{R}^3)).$$ Moreover,
		\begin{align}\label{radial-energy estimate E(t)}
			\sup_{0< t\le T_0} \mathcal{E}(t) \lesssim 1.
		\end{align}
	\end{prop}
	\begin{proof}
		It suffices to show $t^*=\infty$. Assume by contradiction that $t^*<\infty$. Then, by \eqref{radial-a priori estimate for u}, for any $t\in(t_0, t^*)$,
		\begin{align*}
			\|S(-t)u(t)\|_{\mathcal{F}\dot{H}_x^1(\mathbb{R}^3)} \leq C_1(t^*)^{1-\frac{3p}{4}},
		\end{align*}
		for some $C_1>0$. For any $\ep>0$, Proposition \ref{GWP-I:radial-prop H1} with $t=t^*$, $t_1=t^*-\ep$, and $M=C_1(t^*)^{1-\frac{3p}{4}}$ yields that there exists $t_2$ depending on $M$ such that $S(-t)u\in C([t^* - \varepsilon, t^* - \varepsilon + t_2]; \mathcal{F}\dot{H}_x^1(\R^3))$. Notice that $t_2$ is actually independent of $\ep$. Taking $\varepsilon= \frac{t_2}{2}$, we obtain $t^*-\varepsilon + t_2 > t^*$, which contradicts the definition of $t^*$. Thus, $t^*=\infty$. 
		
		The estimate \eqref{radial-energy estimate E(t)} follows from $T^*=\frac{1}{t^*}=0$ and \eqref{radial-E(t)}.
	\end{proof}

	\subsection{Additional regularity}
	From the above argument, we obtain a unique solution of \eqref{NLS} such that
	\begin{align*}
		S(-t)u\in C((0,\infty);\mathcal{F}\dot{H}_x^1(\mathbb{R}^3)).
	\end{align*}
	Next, we improve the regularity of the solution $u$ to $\dot{H}_x^{s_c}(\mathbb{R}^3)$. 
	\begin{prop}
		Let the assumptions in Theorem \ref{main-theorem2} hold. Let also $u$ be the global solution given in the former step such that $S(-t)u\in C((0,\infty); \mathcal{F}\dot{H}_x^1(\mathbb{R}^3))$. Then we have that $S(-t)u\in C((0,\infty); \dot{H}_x^{s_c}\cap \mathcal{F} \dot{H}_x^1(\mathbb{R}^3)) $.
		
	\end{prop}
	\begin{proof}
		It suffices to prove that for any $0<t_0<t$,
		\begin{align}\label{radial-GWP-II:Hs_c}
			\||\nabla|^{s_c}u\|_{L_t^{\infty}L_x^{2}([t_0, t]\times \mathbb{R}^3)}\lesssim_{t}1.
		\end{align}
		Note that $t$ can be viewed as a constant, so we omit its dependence.
		
		We basically follow the argument as in Proposition \ref{P2}. Set $I:=[a, a+b]$, where $t_0\leq a< t$. Assume $u(a)\in \dot{H}_x^{s_c}(\mathbb{R}^3)$. We consider the integral equation
		\begin{align*}
			u(t)=S(t-a)u(a)-i\int_{a}^{t}S(t-\tau)(|u|^pu)\mathrm{d}\tau.
		\end{align*}
		In the following, we restrict the spacetime variable $(t,x)\in I\times \mathbb{R}^3$. By \eqref{radial-a priori estimate for u} and the same argument as in \eqref{radial-inhomogeneous-J(t)u},
		\begin{align*}
			\|J(t)u\|_{L_t^{\infty}L_x^{2}\cap L_t^{2}L_x^6} &\leq C\|S(-a)u(a)\|_{\mathcal{F}\dot{H}_x^1(\R^3)} + Ct_0^{-p}b^{\frac{4-p}{4}}\|J(t)u\|^{p+1}_{L_t^{\infty}L_x^{2}\cap L_t^{2}L_x^6}\\
			&\leq Ct^{1-\frac{3p}{4}} + C b^{\frac{4-p}{4}} \|J(t)u\|^{p+1}_{L_t^{\infty}L_x^{2}\cap L_t^{2}L_x^6}.
		\end{align*}
		Choose $b>0$ sufficiently small, depending on $t$, so that
		\begin{align}\label{radial-GWP-II-estimate-H^s_c}
			\|J(t)u\|_{L_t^{\infty}L_x^{2}\cap L_t^{2}L_x^6} \lesssim 1.
		\end{align}
		For the inhomogeneous term, by Lemma \ref{RSE:radial strichartz estimate} and \eqref{radial-GWP-II-estimate-H^s_c},
		\begin{equation}\label{est-radial-restrict}
			\begin{aligned}
				\left\|\int_{a}^{t}S(t-\tau)(|u|^pu)\mathrm{d}\tau\right\|_{L_t^{\infty}\dot{H}_x^{s_c}\cap L_t^{q_3}L_x^{r_3}} &\lesssim \||u|^pu\|_{L_t^{\wt{q}_3'}L_x^{\wt{r}_3^{\prime}}}\\
				&\lesssim |I|^{\frac{4-p}{4}} \|u\|^p_{L_t^{\infty}L_x^{6}}\|u\|_{L_t^{q_3}L_x^{r_3}}\\
				&\lesssim b^{\frac{4-p}{4}} t_0^{-p} \|J(t)u\|^p_{L_t^{\infty}L_x^{2}}\|u\|_{L_t^{q_3}L_x^{r_3}}\\
				&\lesssim b^{\frac{4-p}{4}}\|u\|_{L_t^{q_3}L_x^{r_3}},
			\end{aligned}
		\end{equation}
		where
		\begin{align*}
			\frac{1}{q_3}=\frac{20-15p}{8p}; \ \ \frac{1}{r_3}=\frac{5p-4}{4p}; \ \ \frac{1}{\wt{q}_3}=\frac{2p^2+15p-20}{8p}; \ \ \frac{1}{\wt{r}_3}=\frac{-2p^2-3p+12}{12p}.
		\end{align*}
		The above exponents are reasonable, since a direct calculation shows that $(q_3, r_3)$ and $(\wt{q}_3,\wt{r}_3)$ are radial-admissible, and satisfy
		\begin{align*}
			\frac{2}{q_3}+\frac{3}{r_3}=\frac{3}{2}-s_c; \quad \frac{2}{\wt{q}_3} + \frac{3}{\wt{r}_3} = \frac{3}{2}+s_c; \quad \frac{1}{\wt{r}_3'}=\frac{p}{6}+\frac{1}{r_3}; \quad \frac{1}{\wt{q}_3'}=\frac{4-p}{4}+\frac{1}{q_3}.
		\end{align*}
		Therefore, 
		\begin{align*}
			\|u\|_{L_t^{\infty}\dot{H}_x^{s_c}\cap L_t^{q_3}L_x^{r_3}}\leq C \|u(a)\|_{\dot{H}_x^{s_c}(\R^3)} + Cb^{\frac{4-p}{4}}\|u\|_{L_t^{\infty}\dot{H}_x^{s_c}\cap L_t^{q_3}L_x^{r_3}}.
		\end{align*}
		Thus, there exists $b$ depending on $t$ such that $Cb^{\frac{4-p}{4}}<\frac{1}{2}$ and 
		\begin{align*}
			\||\nabla|^{s_c}u\|_{L_t^{\infty} L_x^2}\lesssim \|u(a)\|_{\dot{H}_x^{s_c}(\R^3)}.
		\end{align*}
		
		Split $[t_0,t]$ into small consecutive intervals $I_j$, with $j=1,2,\cdots, m$, such that $|I_j|\leq b$. \eqref{radial-GWP-II:Hs_c} follows by induction on $I_j$.
	\end{proof}
	\begin{remark}\label{remark-radial-restriction for p}
		When estimating \eqref{est-radial-restrict}, we expect to bound the term $\||u|^p u\|_{L_t^{\wt{q}^{\prime}}L_x^{\wt{r}^{\prime}}}$ by $|I|^{\frac{1}{q_0}}\|J(t) u\|^p_{L_t^{q}L_x^r}\|u\|_{L_t^{q_1}L_x^{r_1}}$, where $(\wt{q}, \wt{r})$ and $(q_1, r_1)$ are radial-admissible, $(q, r)$ is admissible, and $q_0\geq 1$. To ensure these conditions hold, we require that $p$ satisfy $2p^2+15p-20\geq 0$, that is, $p\geq \gamma_1$.
	\end{remark}
	
	Similar to Lemma \ref{lem--s_c-global-bound}, we have the following spacetime estimate through pseudo conformal energy.
	\begin{lem}\label{lem-radial-global bound}
		Let the assumptions in Theorem \ref{main-theorem2} hold. Let also $u$ be the global solution given in the former step such that $S(-t)u\in C((0,\infty);\dot{H}_x^{s_c}\cap \mathcal{F}\dot{H}_x^1(\mathbb{R}^3))$. Assume that $(q_1, r_1)$ is given in \eqref{radial-LWP-I-(q1,r1)}. Then
		\begin{align}\label{radial-global-estimate-p+2}
			\|u\|_{L_t^{q_1} L_x^{r_1} ((0,\infty)\times \mathbb{R}^3)} \lesssim 1.
		\end{align}
	\end{lem}
	\begin{proof}
		It follows from Proposition \ref{radial:LWP-I} that 
		\begin{align}
			\|u\|_{L_t^{q_1} L_x^{r_1} ((0,t_0)\times \mathbb{R}^3)} \lesssim 1.
		\end{align}
		\eqref{radial-global-estimate-p+2} will hold once we prove that
		\begin{align}\label{radial-bound:> t_0}
			\|u\|_{L_t^{q_1} L_x^{r_1} ((t_0,\infty)\times \mathbb{R}^3)} \lesssim 1.
		\end{align}
		By \eqref{PCT-remark-2}, \eqref{radial-bound:> t_0} is equivalent to 
		\begin{align}\label{radial-bound:< T_0}
			\|t^{\frac{3}{2}-\frac{2}{p}}\mathcal{U}\|_{L_t^{q_1} L_x^{r_1} ((0,T_0)\times \mathbb{R}^3)} \lesssim 1,
		\end{align}
		where $T_0=t_0^{-1}$. Denote $I:=(0,T_0).$ Multiply both sides of \eqref{Pseudo-conformal-energy} by $t^{(\frac{3}{2}-\frac{2}{p})r_1}$,  
		\begin{align*}
			t^{(\frac{3}{2}-\frac{2}{p})r_1} \mathcal{E}(t) = \frac{1}{4}t^{3 - \frac{4}{p}}\|\nabla\mathcal{U}\|^2_{L_x^2(\R^3)} + \frac{1}{p+2} \big(t^{\frac{3}{2}-\frac{2}{p}}\|\mathcal{U}\|_{L_x^{r_1}(\R^3)}\big)^{r_1} .
		\end{align*}
		Therefore, 
		\begin{align}\label{radial-global-estimate->t_0-1}
			\|t^{\frac{3}{2}-\frac{2}{p}}\mathcal{U}\|^{r_1}_{L_t^{q_1} L_x^{r_1} (I\times \mathbb{R}^3)} \lesssim \|t^{(\frac{3}{2}-\frac{2}{p})r_1} \mathcal{E}(t) \|_{L_t^{\frac{2p}{4-p}}(I)} + \| t^{3 - \frac{4}{p}}\|\nabla\mathcal{U}\|^2_{L_x^2(\R^3)} \|_{L_t^{\frac{2p}{4-p}}(I)}.
		\end{align}
		By \eqref{radial-energy estimate E(t)} and the fact that $(\frac{3}{2}-\frac{2}{p})r_1\cdot \frac{2p}{4-p} >-1$,
		\begin{align}\label{radial-global-estimate->t_0-2}
			\|t^{(\frac{3}{2}-\frac{2}{p})r_1} \mathcal{E}(t) \|_{L_t^{\frac{2p}{4-p}}(I)} \lesssim \|t^{(\frac{3}{2}-\frac{2}{p})r_1} \|_{L_t^{\frac{2p}{4-p}}(I)} \lesssim 1.
		\end{align}
		Under the assumptions of Theorem \ref{main-theorem2}, \eqref{estimate for E(T_0)} also holds. We will use a special version of \eqref{estimate for E(T_0)}, that is 
		\begin{align}\label{radial-estimate for E(T_0)}
			\int_{0}^{T_0}\left(t^{\frac{1}{2}(2-\frac{3p}{2})}\|\nabla\mathcal{U}\|_{L_x^2(\R^3)}\right)^2\frac{\mathrm{d}t}{t} = \int_{0}^{T_0}\big( t^{\frac{2-3p}{4}}\|\nabla\mathcal{U}\|_{L_x^2(\R^3)} \big)^2\mathrm{d} t \lesssim 1.
		\end{align}
		By H\"older's inequality, \eqref{radial-estimate for E(T_0)}, and the fact that $3p^2+4p-8>0$ under the assumptions of Theorem \ref{main-theorem2},
		\begin{equation}\label{radial-global-estimate->t_0-3}
			\begin{aligned}
				\big\| t^{3 - \frac{4}{p}}\|\nabla\mathcal{U}\|^2_{L_x^2(\R^3)} \big\|^{\frac{2p}{4-p}}_{L_t^{ \frac{2p}{4-p}}(I)} &= \int_{0}^{T_0} t^{\frac{3p^2+4p-8}{4-p}} \big( t^{\frac{2-3p}{4}}\|\nabla\mathcal{U}\|_{L_x^2(\R^3)} \big)^{\frac{4p}{4-p}} \mathrm{d} t\\
				&\lesssim \int_{0}^{T_0}\big( t^{\frac{2-3p}{4}}\|\nabla\mathcal{U}\|_{L_x^2(\R^3)} \big)^2\mathrm{d} t \lesssim 1.
			\end{aligned}
		\end{equation}
		By \eqref{radial-global-estimate->t_0-1}, \eqref{radial-global-estimate->t_0-2} and \eqref{radial-global-estimate->t_0-3}, we obtain \eqref{radial-bound:< T_0} and hence \eqref{radial-global-estimate-p+2}.
	\end{proof}
	
	\begin{prop}\label{prop-radial-s_c-1-global bound}
		Let the assumptions in Theorem \ref{main-theorem2} hold. Let also $u$ be the global solution given in the former step. Then for any radial-admissible pair $(q,r)$ satisfying $\frac{2}{q}+\frac{3}{r}= \frac{3}{2}-s_c$,
		\begin{align}\label{est-radial-s_c-1-global bound}
			\|u\|_{L_t^{\infty}\dot{H}_x^{s_c}\cap L_t^{q}L_x^{r}((0,\infty)\times\mathbb{R}^3)}+\|J(t)u\|_{L_t^{\infty}L_x^2\cap L_t^2L_x^6((0,\infty)\times\R^3)} \lesssim 1.
		\end{align}
	\end{prop}
	\begin{proof}
		Let $(q_1, r_1)$, $(\wt{q}_1, \wt{r}_1)$ and $(q_2, r_2)$ be as defined in \eqref{radial-LWP-I-(q1,r1)}. By Lemmas \ref{RSE:radial strichartz estimate} and \ref{lem-radial-global bound},
		\begin{equation}
			\begin{aligned}\label{est-radial-s_c-global bound}
				\|u\|_{L_t^{\infty}\dot{H}_x^{s_c}\cap L_t^{q}L_x^{r}((0,\infty)\times\mathbb{R}^3)} &\lesssim \|u_0\|_{\dot{H}_x^{s_c}(\R^3)} +  \||u|^pu\|_{L_t^{\wt{q}'_1}L_x^{\wt{r}'_1}((0,\infty)\times\mathbb{R}^3)}\\
				&\lesssim \|u_0\|_{\dot{H}_x^{s_c}(\R^3)} + \|u\|^{p+1}_{L_t^{q_1} L_x^{r_1}((0,\infty)\times\mathbb{R}^3)}\lesssim 1. 
			\end{aligned}
		\end{equation}
		For $I=(\tau_0, \tau_1)$ or $(\tau_0,\infty)$, applying the same argument as in \eqref{radial-LWP-J(t)},
		\begin{align*}
			\|J(t)u\|_{L_t^{\infty}L_x^2\cap L_t^2L_x^6(I\times\R^3)}&\lesssim \|S(-\tau_0)u(\tau_0)\|_{\mathcal{F}\dot{H}_x^1(\mathbb{R}^3)} \\
			&\qquad + \|u\|^{p}_{L_t^{q_1} L_x^{r_1}(I\times\mathbb{R}^3)}\|J(t)u\|_{L_t^{\infty}L_x^2\cap L_t^2L_x^6(I\times\R^3)}.
		\end{align*}
		Similar to Lemma \ref{Lemma-short interval bound} and Proposition \ref{prop--s_c-1-gloabl bound}, there exists a partition $t_0=\tau_0<\tau_1<\cdots<\tau_{m-1}<\tau_{m}=\infty$ such that 
		\begin{align}\label{prop-radial-global bound-1}
			\|J(t)u\|_{L_t^{\infty}L_x^2\cap L_t^2L_x^6(I_j\times\R^3)}\lesssim_{\tau_j} 1,
		\end{align}
		where $I_j=[\tau_j, \tau_{j+1})$, $0\leq j < m$. Note that $\tau_j$, $0\leq j<m$ can be viewed as constants. Thus, by \eqref{prop-radial-lwp-I} and \eqref{prop-radial-global bound-1},
		\begin{align}\label{est-radial-1-global bound}
			\|J(t)u\|_{L_t^{\infty}L_x^2\cap L_t^2L_x^6((0,\infty)\times\R^3)}\lesssim 1,
		\end{align}
		\eqref{est-radial-s_c-1-global bound} follows from \eqref{est-radial-s_c-global bound} and \eqref{est-radial-1-global bound}.
	\end{proof}

	\subsection{Scattering} 
	\begin{prop}[Scattering]
		Let $u$ be the unique global solution to \eqref{NLS} given in Theorem \ref{main-theorem2}. Then the solution $u$ scatters in $\dot{H}_x^{s_c}\cap\mathcal{F}\dot{H}_x^1(\mathbb{R}^3)$, that is, there exist $u_{\pm}\in \dot{H}_x^{s_c}\cap\mathcal{F}\dot{H}_x^1(\mathbb{R}^3)$ such that 
		\begin{align}\label{radial-scattering}
			\lim_{t\rightarrow \pm\infty}\|S(-t)u-u_{\pm}\|_{\dot{H}_x^{s_c}\cap\mathcal{F}\dot{H}_x^1(\mathbb{R}^3)}=0.
		\end{align}  
	\end{prop}
	\begin{proof}
		The proof is similar to Proposition \ref{scattering}. We consider only the case $t\rightarrow +\infty$, since the case $t\rightarrow -\infty$ is treated in the same way. Let $(q_1, r_1)$, $(\wt{q}_1, \wt{r}_1)$ and $(q_2, r_2)$ be as defined in \eqref{radial-LWP-I-(q1,r1)}.
		
		It follows from Proposition \ref{prop-radial-s_c-1-global bound} that for any $0<\varepsilon \ll 1$, there exists $\tau_0>1$ such that 
		\begin{align}\label{radial-scattering-bound for u}
			\|u\|_{L_t^{q_1}L_x^{r_1}((\tau_0,\infty)\times\mathbb{R}^3)} + \|J(t)u\|_{L_t^{q_2}L_x^{r_2}((\tau_0,\infty)\times\mathbb{R}^3)} < \varepsilon.
		\end{align}
		For any $\tau_0<\tau_1<\tau_2$, by \eqref{scattering-S(-t)u(t)} and Lemma \ref{RSE:radial strichartz estimate},
		\begin{align*}
			\|S(-\tau_2)u(\tau_2)-S(-\tau_1)u(\tau_1)\|_{\dot{H}_x^{s_c}(\mathbb{R}^3)}\lesssim\||u|^pu\|_{L_t^{\wt{q}'_1}L_x^{\wt{r}'_1}((\tau_1,\tau_2)\times\mathbb{R}^3)}\lesssim \|u\|^{p+1}_{L_t^{q_1} L_x^{r_1}((\tau_1,\tau_2)\times\mathbb{R}^3)}. 
		\end{align*}
		By the Strichartz estimate,
		\begin{align*}
			\|S(-\tau_2)u(\tau_2)-S(-\tau_1)u(\tau_1)\|_{\mathcal{F}\dot{H}_x^{1}(\mathbb{R}^3)}&=\left\|\int_{\tau_1}^{\tau_2}S(-\tau)J(\tau)(|u|^pu)\mathrm{d}\tau\right\|_{L_x^2(\mathbb{R}^3)}\\
			&\lesssim\|J(t)(|u|^pu)\|_{L_t^{q'_2}L_x^{r'_2}((\tau_1,\tau_2)\times\mathbb{R}^3)}\\
			&\lesssim \|u\|^{p}_{L_t^{q_1} L_x^{r_1}((\tau_1,\tau_2)\times\mathbb{R}^3)}  \|J(t)u\|_{L_t^{q_2}L_x^{r_2}((\tau_1,\tau_2)\times\mathbb{R}^3)}.
		\end{align*}
		Therefore, by \eqref{radial-scattering-bound for u}, 
		\begin{align}\label{radial-scattering-FH1}
			\|S(-\tau_2)u(\tau_2)-S(-\tau_1)u(\tau_1)\|_{\dot{H}_x^{s_c}\cap\mathcal{F}\dot{H}_x^1(\mathbb{R}^3)}\lesssim \varepsilon.
		\end{align}
		\eqref{radial-scattering} follows from \eqref{radial-scattering-FH1}.
	\end{proof}
	
	\bigskip
	\section*{Acknowledgment}
	J. Shen is partially supported by the National Key R$\&$D Program of China, 2025YFA1018500; NSFC 12501617; Fundamental Research Funds for the Central Universities, Nankai University (63261060); Natural Science Foundation of Tianjin (24JCQNJC02060).

\end{document}